\documentclass[a4paper]{scrartcl}
\usepackage[]{amsmath,amssymb}
\usepackage{amsthm}
\usepackage{enumitem}
\usepackage[utf8]{inputenc}
\usepackage[T1]{fontenc}
\usepackage{lmodern}
\usepackage{fourier}
\usepackage{graphicx}
\usepackage{lmodern}
\usepackage{dsfont}
\usepackage{amsmath} 
\usepackage{mathrsfs}
\usepackage{tikz}      
\usepackage{tikz-cd}   
\usepackage{mathtools}
\usepackage{hyperref}

\DeclareMathOperator{\oD}{d}
\def\D{\oD\!}
\def\dif{\,\D}

\usepackage{tikz}
\def\R{{\mathbf R}}
\def\Rpd{{\mathbf R_{+}^{d}}}
\def\Rps{\R_{+}^{*}}
\def\S{{\mathbf S}}

\def\N{{\mathbf N}}

\def\car#1{{\mathbf 1}}

\newcommand{\esp}[1]{{\mathbf E}\left[#1\right]}

\def\L{\mathbf L}

\def\B{{\mathcal B}}
\def\/{\,|\,} 
\newcommand{\var}{\operatorname{Var}}

\def\car{\mathbf 1}

\usepackage{amsmath,amsfonts,amssymb,amsthm} 
\newtheoremstyle{hermesexercises}{11pt}{11pt}{\normalfont}{0pt}{\scshape}{.--}{.5em}{}
{\theoremstyle{hermesexercises}
}
\newtheoremstyle{hermesremark}{11pt}{11pt}{\normalfont}{0pt}{\scshape}{.--}{.5em}{}
\newtheoremstyle{myhermesremark}{11pt}{11pt}{\normalfont}{0pt}{\scshape}{.--}{.5em}{}
{\theoremstyle{myhermesremark}
  \newtheorem{remark}{Remark}
\newtheorem{example}{Example}

}
\newtheoremstyle{hermestheorem}{11pt}{11pt}{\normalfont}{0pt}{\scshape}{.--}{.5em}{}
{\theoremstyle{hermestheorem}
\newtheorem{theoremT}{Theorem}
\newtheorem{definitionT}{Definition}
\newtheorem{corollaryT}[theoremT]{Corollary}
\newtheorem{lemmaT}[theoremT]{Lemma}}
\usepackage{tcolorbox}
\tcbuselibrary{skins,xparse,breakable}
\definecolor{fonce}{HTML}{323031}
\definecolor{cadet}{HTML}{084C61}
\definecolor{vertdeau}{HTML}{177E89}
\definecolor{ocre}{HTML}{FFC857}
\definecolor{creme}{HTML}{DB3A34}

\usepackage{bm}
\usepackage{leftindex}

\tcbset{textmarker/.style={%
skin=enhancedmiddle jigsaw,breakable,parbox=false, boxrule=0mm,leftrule=5mm,rightrule=5mm,boxsep=0mm,arc=0mm,outer arc=0mm, left=3mm,right=3mm,top=1mm,bottom=1mm,toptitle=1mm,bottomtitle=1mm,oversize}}

\tcbset{textmarker2/.style={%
skin=enhancedmiddle jigsaw,breakable,parbox=false, boxrule=0mm,leftrule=3mm,boxsep=0mm,arc=0mm,outer arc=0mm, left=3mm,right=3mm,top=1mm,bottom=1mm,toptitle=1mm,bottomtitle=1mm,oversize}}

\tcolorboxenvironment{lemmaT}{textmarker2,colback=ocre!5!white,colframe=ocre}
\newenvironment{lemma}{%
     \begin{lemmaT}}{\end{lemmaT}}

   \tcolorboxenvironment{definitionT}{textmarker2,colback=creme!5!white,colframe=creme}
   \newenvironment{definition}{%
     \begin{definitionT}}{\end{definitionT}}
      \tcolorboxenvironment{theoremT}{textmarker2,colback=vertdeau!5!white,colframe=vertdeau}
 \newenvironment{theorem}{%
     \begin{theoremT}}{\end{theoremT}}
      \tcolorboxenvironment{corollaryT}{textmarker2,colback=vertdeau!5!white,colframe=vertdeau}
 \newenvironment{proposition}{%
     \begin{theoremT}}{\end{theoremT}}
      \tcolorboxenvironment{corollaryT}{textmarker2,colback=vertdeau!5!white,colframe=vertdeau}
      \newenvironment{corollary}{%
     \begin{corollaryT}}{\end{corollaryT}}
\usepackage{enumitem}
\usetikzlibrary{%
  arrows,%
  calc,%
  fit,%
  patterns,%
  plotmarks,%
  shapes.geometric,%
  shapes.misc,%
  shapes.symbols,%
  shapes.arrows,%
  shapes.callouts,%
  shapes.multipart,%
  shapes.gates.logic.US,%
  shapes.gates.logic.IEC,%
  er,%
  automata,%
  backgrounds,%
  chains,%
  topaths,%
  trees,%
  petri,%
  mindmap,%
  matrix,%
  calendar,%
  folding,%
  fadings,%
  through,%
  positioning,%
  scopes,%
  decorations.fractals,%
  decorations.shapes,%
  decorations.text,%
  decorations.pathmorphing,%
  decorations.pathreplacing,%
  decorations.footprints,%
  decorations.markings,%
  shadows}
\newcommand{\iid}{\textnormal{i.i.d.\ }}
\newcommand{\lbra}{[\![}
\newcommand{\rbra}{]\!]}
\renewcommand{\esp}{\mathbb{E}}
\renewcommand{\var}{\mathbb{V}}
\newcommand{\Cov}{\mathbb{C}\mathrm{ov}}
\newcommand{\prob}{\mathbb{P}}
\newcommand{\ind}{\mathds{1}}

\newcommand{\lb}{\{}
\newcommand{\rb}{\}}

\newcommand\numberthis{\addtocounter{equation}{1}\tag{\theequation}}

\newcommand{\dom}{\operatorname{Dom}}

\def\eqdis{\overset{\mathrm{d}}{=}}

\def\pms{\mathbb{P}}

\def\Pa{\mathbf{P}^{\alpha}}
\def\La{\mathscr{L}_{\alpha}}
\def\Da{\mathbf{D}_{\alpha}}

\def\Pi{\mathbf{P}^{1}}
\def\Li{\mathscr{L}_{1}}
\def\Di{\mathbf{D}_{1}}

\def\Pan{\mathbf{P}^{\alpha, \nu}}
\def\Lan{\mathscr{L}_{\alpha, \nu}}
\def\Dan{\mathbf{D}_{\alpha, \nu}}

\def\Panbm{\mathbf{P}^{\bm{\alpha}, \nu}}
\def\Lanbm{\mathscr{L}_{\bm{\alpha}, \nu}}
\def\Danbm{\mathbf{D}_{\bm{\alpha}, \nu}}

\def\Pn{\mathbf{P}^{1,\nu}}
\def\Ln{\mathscr{L}_{1,\nu}}
\def\Dn{\mathbf{D}_{1,\nu}}

\def\cic{\mathcal{C}_{c}^{1}(\Eos)}

\def\cilog{\mathcal{C}_{\text{log}}^{1}(\Eos)}
\def\cilogi{\mathcal{C}_{\text{log}}^{1}(\Rps)}
\def\Epol{E_{\mathrm{pol}}}
\def\El{E_{\bm{\ell}}}
\def\Eo{E_{\bm{0}}}
\def\Els{E^{*}_{\bm{\ell}}}
\def\Eos{E^{*}_{\bm{0}}}
\def\m{\bm{m}}
\def\eqdis{\overset{\mathrm{d}}{=}}

\newcommand{\mxint}[2]{\prescript{\mathrlap{e}}{}{\int_{#1}^{#2}}}

\usepackage[backend=bibtex,style=numeric]{biblatex}
\title{Functional analysis of multivariate max-stable distributions}
\author{B. Costacèque \and L. Decreusefond}
\date{2025}
\begin{document}
\maketitle{}
\begin{abstract}
	We study the connections existing between max-infinitely divisible distributions and Poisson processes from the point of view of functional analysis. More precisely, we derive functional identities for the former by using well-known results of Poisson stochastic analysis. We also introduce a family of Markov semi-groups whose stationary measures are the so-called multivariate max-stable distributions. Their generators thus provide a functional characterization of extreme valued distributions in any dimension. Additionally, we give a few functional identities associated to those semi-groups, namely a Poincaré identity and commutation relations. Finally, we present a stochastic process whose semi-group corresponds to the one we introduced and that can be expressed using extremal stochastic integrals.
  %
\end{abstract}



\noindent{Math Subject Classification: 39B62, 47D07,60E07}
\setcounter{tocdepth}{1}
\tableofcontents

\section{Introduction}

Stochastic modeling is frequently grounded in the theory of Markov processes, which are characterized primarily by their infinitesimal generator~\cite{Ethier86}. According to the Hille-Yosida theorem, the dynamics of a Markov process are fully determined by its associated semi-group. In practice, the stationary distribution, when it exists, plays a central role, as it describes the long-term behavior of the process. From a more formal perspective, it is well known that, given any one of the following three objects (a Markov process, a generator satisfying the Hille-Yosida conditions, or a strongly continuous semi-group on a Banach space) one can, at least abstractly, construct the other two~\cite{Ma_1992}. Associated to this triptych are the Dirichlet form and the carré du champ operator~\cite{Bouleau1991,fukushima2011dirichlet}, which open the way to potential theory. These concepts are fundamental in the analysis and geometry of Markov diffusion processes, as developed in~\cite{Bakry14}. The so-called $\Gamma$-calculus, detailed in this reference, leads to fundamental functional inequalities (such as the Poincaré and log-Sobolev inequalities) and to concentration inequalities for the stationary measure. As emphasized in the introduction of~\cite{Bakry14} and clearly explained in~\cite{chafai2004entropies}, the techniques developed therein rely crucially on the locality and symmetry of the semi-group with respect to the stationary measure, as well as on the diffusion property, which ensures that the carré du champ is a true derivation.

Stochastic quantization, initially introduced by physicists~\cite{Parisi81,Namiki92}, addresses the inverse problem: given a probability measure, one seeks a Markov process for which this measure is stationary. This approach offers the possibility of deriving functional inequalities for the chosen measure using the techniques of~\cite{Bakry14}. This was one of the main of the two motivations for the present work. Our  initial motivation stemmed from considerations related to Stein's method. In its modern formulation (see~\cite{Decreusefond15,Decreusefond22}), this method is based on the identity
\begin{equation}
  \label{eq_dalphas_core:1}
  \int_{E} f\,\mathrm{d}\mu - \int_{E} f\,\mathrm{d}\nu = \int_{E} \int_{0}^{\infty} L P_{t}f\,\mathrm{d}t\,\mathrm{d}\nu,
\end{equation}
where $L$ is the generator of the semi-group associated to the target measure~$\mu$ by quantization, and $\nu$ is any other probability measure on $(E, \mathcal{E})$. For the standard Gaussian measure on $\R^{n}$, the classical operator is
\begin{equation*}
  Lf(x) = -\langle x, \nabla f(x)\rangle + \Delta f(x),
\end{equation*}
with the associated Ornstein-Uhlenbeck semi-group given by
\begin{equation*}
  P_{t}f(x) = \int_{\R^{n}} f\bigl(e^{-t}x + \sqrt{1-e^{-2t}}\,y\bigr)\,\mathrm{d}\mu(y).
\end{equation*}
If $\mu$ denotes the law of an $\alpha$-stable distribution on $\R$, the corresponding semi-group is
\begin{equation*}
  P_{t}f(x) = \int_{\R} f\big(e^{-t/\alpha}x + (1-e^{-t})^{1/\alpha}y\big)\,\mathrm{d}\mu(y),
\end{equation*}
with generator
\begin{equation*}
  Lf(x) = -\frac{1}{\alpha} x f'(x) + \Delta^{\alpha/2} f(x),
\end{equation*}
where $\Delta^{\alpha/2}$ denotes the fractional Laplacian (see~\cite{Chen22,Xu19}).

A crucial observation is that the semi-group property of these operators is a direct consequence of the stability property of the underlying measures: for any $\alpha$-stable law with $\alpha \in (0,2]$,
\begin{equation}
  \label{eq_coupon_core:3}
  a X' + b X'' \stackrel{\mathrm{d}}{=} X,
\end{equation}
where $X', X''$ are independent copies of $X$, for any $a,b\ge 0$ such that $a^\alpha+b^\alpha=1$. The case $\alpha=2$ corresponds to the Gaussian distribution. Formally, equation~\eqref{eq_coupon_core:3} can be written as
\begin{equation*}
  D_{a}X' \oplus D_{b}X'' \stackrel{\mathrm{d}}{=} X,
\end{equation*}
where $D_{a}$ denotes multiplication by $a$, and $\oplus$ is ordinary addition. The algebraic structure here is that of a semi-group (addition) together with a group $(D_{a})_{a \in T}$ of automorphisms satisfying $D_{a} \circ D_{b} = D_{ab}$. In the seminal work~\cite{Davydov08}, such a structure is called a convex cone. It is shown there that many other examples of stable distributions arise by changing the meaning of $\oplus$ and the group of automorphisms. These distributions are of interest because their stability implies their appearance in various limit theorems. In this work, we focus on max-stable distributions, motivated by their wide range of applications in fields such as meteorology, hydrology, epidemiology, and finance. In light of the preceding discussion, this leads us to consider the semi-group defined by
\begin{equation*}
  P_{t}f(x) = \int_{E} f\big(D_{e^{-t/\alpha}}x \oplus D_{(1-e^{-t})^{1/\alpha}}y\big)\,\mathrm{d}\mu_{\alpha}(y),
\end{equation*}
where $\mu_{\alpha}$ is an $\alpha$-max-stable distribution. We can then compute
the generator, carré du champ operator, and Dirichlet form associated with this
semi-group, and even identify the underlying Markov process using the notion of
stochastic extremal integral. However, the resulting Dirichlet form is neither
local, symmetric, nor diffusive, so the full machinery developed
in~\cite{Bakry14} is not directly applicable. Nevertheless, a fundamental result
states that a random variable with a stable law can be represented as a
functional of a marked Poisson point process (see~\eqref{prelim_LePage_1}
below). This identity, known as the de Haan-LePage representation, allows us to
leverage functional identities for the Poisson process and to establish Poincaré
and log-Sobolev inequalities for max-stable distributions. We first address the
multivariate setting, which is significantly more intricate than the univariate
case. The former's properties are strongly influenced by the spectral measure.

Our approach has some resemblance to \cite{Arras19b,Houdré98},
which examines univariate infinitely divisible random variables through the lens
of the Lévy-Khinchin formula. These two papers primarily focus on the covariance
representation (see~\eqref{3_cov_id_max2} for our version) in the univariate
context, which they apply to  Stein's method. We here start from the relation
given by the stability hypothesis and analyse deeply the  structure of the
Dirichlet space associated to max stable random variables.

The remainder of this paper is organized as follows. Section~\ref{sec:preliminaries} introduces the notations and preliminary results required for the sequel. Section~\ref{sec:stoch-analys-max} explores the connections between max-infinitely divisible random vectors and stochastic analysis for Poisson processes. Section~\ref{sec:max-stable-ornstein} presents the max-stable analogue of the Ornstein-Uhlenbeck semi-group $(P_{t})_{t\geq 0}$ and investigates its properties.

\section{Preliminaries}
\label{sec:preliminaries}
\subsection{Max-stable and max-id random variables}

The set of integers between $n$ and $m$ is denoted by $\lbra n,m \rbra$. Let
$\bm{x} = (x^{1},\dots, x^{d})$ and $\bm{y} = (y^{1},\dots, y^{d})$ be two
vectors in $\R^{d}$, with $x^{j} \leq y^{j}$ for all $j \in \lbra 1 ,d \rbra$.
We set:
\[ [\bm{x}, \bm{y}] \coloneq \prod_{j=1}^{d}[x^{j}, y^{j}].
\]
Likewise, we take $[\bm{x}, \bm{y}) \coloneq \prod_{j=1}^{d}[x^{j}, y^{j})$. Let
$\El$ be the set of vectors in $[\bm{\ell}, +\bm{\infty})$, minus $\bm{\ell}$
itself:
\[
  \El \coloneq [\bm{\ell}, +\bm{\infty}) \setminus \lb \bm{\ell} \rb.
\]
We will also need to work with the vectors $\bm{x}$ that are strictly greater
than $\bm{\ell}$, in the sense that $x^{j} > \ell^{j}$ for all
$j \in \lbra 1,d \rbra$. We denote the set of such vectors by:
\[
  \Els \coloneq (\bm{\ell}, +\bm{\infty}).
\]
In the sequel, the notation $\bm{x} \leq \bm{y}$ means that the coordinates
$x^{j}$ of $\bm{x}$ are less than or equal to their corresponding coordinates
$y^{j}$ of $\bm{y}$, while $\bm{x} \nleq \bm{y}$ signifies that at least one
coordinate of $\bm{x}$ is greater than its counterpart of $\bm{y}$. The
following notations come from tropical geometry:
\[
  \bm{x}\oplus\bm{y} = \big(\max(x^{1}, y^{1}),\dots, \max(x^{d}, y^{d})\big)
\]
and
\[
  \bm{x}\odot\bm{y} = \big(\min(x^{1}, y^{1}),\dots, \min(x^{d}, y^{d})\big).
\]
Besides $\max \bm{x} \coloneq \max(x^{1},\dots,x^{d})$ (respectively
$\min \bm{x} \coloneq \min(x^{1},\dots,x^{d})$) denotes the greatest coordinate
(respectively least) of $\bm{x}$. Consequently, it is always a scalar.

We say that a random vector $\bm{Z}$ is \textit{max-stable} if for all vectors
$\bm{a}$, $\bm{b}$ in $\R_{+}^{d}$, there exists
$\bm{c}, \bm{d} \in \R_{+}^{d}$, such that
\begin{align}\label{prelim_ms_hyp}
  \bm{a}\bm{Z}\oplus \bm{b}\bm{Z}' \eqdis \bm{c}\bm{Z} + \bm{d},
\end{align}
where $\bm{Z}'$ is an \iid copy of $\bm{Z}$. In \eqref{prelim_ms_hyp}, the sum
and the multiplication between vectors are defined in a coordinate-wise way. A
basic result in extreme value theory (see \cite{Resnick87} or \cite{Embrechts13}
for instance) states that the marginals $Z^{j}$ of such a random vector
$\bm{Z} = (Z^{1},\dots,Z^{d})$ are necessarily either Fréchet, Gumbel or Weibull
random variables. The Fréchet distribution $\mathcal{F}(\alpha, \sigma)$ with
shape parameter $\alpha > 0$ and scale parameter $\sigma > 0$ has c.d.f.
\begin{align}\label{prelim_Fréchet}
  F(x) = \begin{cases}
           e^{-\big(\frac{\sigma}{x}\big)^{\alpha}} & \quad \text{if\ } x \geq 0 \\
           0                                        & \quad \text{otherwise.}
         \end{cases}
\end{align}
When $\sigma = 1$, we will simply note $\mathcal{F}(\alpha)$. In the sequel, we
will assume that the $Z^{j}$ all have the same Fréchet distribution
$\mathcal{F}(\alpha)$ for some $\alpha > 0$. When $\alpha = 1$, it is common to
call such a random vector \textit{simple}. We will keep using this terminology
for max-stable vectors whose marginals all have the same Fréchet
$\mathcal{F}(\alpha)$ distribution. Simple max-stable random vectors have
support on $\Eos$ and satisfy:
\begin{align}\label{prelim_stability}
  \bm{a}\bm{Z}\oplus \bm{b}\bm{Z}' \eqdis \big(\bm{a}^{\alpha} + \bm{b}^{\alpha})\bm{Z},
\end{align}
where $\bm{x}^{\alpha}$ must be understood in a component-wise manner. We say a
Radon measure $\mu$ on $\Eo$ possesses \textit{the $\alpha$-homogeneity
  property} if for all $t>0$:
\begin{align}\label{prelim_homo}
  \mu\big(t^{\frac{1}{\alpha}}B\big) = t^{-1}\mu(B),\ B \in B(\Eo),
\end{align}
where $\B(\Eo)$ denotes the Borel $\sigma$-field of $\Eo$. Note that a Radon
measure on $\Eo$ is $\sigma$-finite. We then have the most important theorem:

\begin{theorem}[de Haan-LePage representation]
  Let $\alpha > 0 $ and $\bm{Z}$ a max-stable random vector with Fréchet
  $\mathcal{F}(\alpha)$ marginals. Then there exists
  $\eta = (\bm{y}_{i})_{i\geq 1}$ a Poisson process on $\Eo$ with
  intensity measure $\mu$ such that the following equality in distribution
  holds:
  \begin{align}\label{prelim_LePage_1}
    \bm{Z} \eqdis \bigoplus_{i = 1}^{\infty}\bm{y}_{i}.
  \end{align}
\end{theorem}

In the sequel, $\mu$ will be called the \textit{exponent measure} of $\bm{Z}$.
We refer to \cite{Last17}, \cite{Privault09} and the references therein for more
about the Poisson process.

Thanks to the so-called \textit{polar decomposition}, it is possible to give
more information about $\mu$. Fix a norm $\Vert \cdot \Vert$ on $\R^{d}$
(henceforth called the \textit{reference norm}) and set
$\Epol \coloneq \Rps \times \S_{+}^{d-1}$, where $\S_{+}^{d-1}$ is the positive
orthant of the sphere with respect to $\Vert\cdot\Vert$, \textit{i.e.}
\[
  \S_{+}^{d-1} \coloneq \big\lb \bm{x} \in \R_{+}^{d},\ \Vert \bm{x} \Vert = 1\big\rb.
\]
For simplicity, we will assume that $\Vert\cdot\Vert$ is normalized so that
$\S_{+}^{d-1} \subseteq [0,1]^{d}$. Define the transformation $T$
\begin{align*}
  \begin{array}{ccccc}
    T & : & \Rps \times \S_{+}^{d-1} & \to     & \Eo                         \\
      &   & (r, \bm{u})              & \mapsto & r\bm{u}^{\frac{1}{\alpha}}.
  \end{array}
\end{align*}
Let  $\mu$ be a measure on $\Eo$, as stated in \cite{Resnick87} (proposition
5.11),  there exists $\nu$ a finite measure on $\S_{+}^{d-1}$
satisfying
\begin{equation}\label{prelim_moment_constraints}
  \int_{\S^{d-1}_{+}}u^{j}\dif\nu(\bm{u}) = 1,\ j \in \lb 1,\dots,d \rb.
\end{equation}
and such that
\begin{align}\label{prelim_polar}
  \mu = T_{*}(\rho_{1}\otimes \nu)
\end{align}
where the right-hand side denotes the pushforward measure of
$\rho_{1}\otimes \nu$ by $T$ and $\rho_{\alpha}$ is the measure on $\Rps$
defined by
\begin{align}\label{prelim_rho}
  \rho_{\alpha}[x, +\infty) \coloneq \frac{1}{x^{\alpha}}\cdotp
\end{align}
Equation~\eqref{prelim_polar} is called the polar decomposition of~$\mu$.
The previous result has the following consequence on the de Haan representation:
there exists a marked Poisson process $\eta = ((r_{i}, \bm{u}_{i}))_{i\geq 1}$
on $\Epol$ such that
\begin{align}\label{prelim_LePage_2}
  \bm{Z} \eqdis \bigoplus_{i=1}^{\infty}r_{i}\bm{u}_{i}^{\frac{1}{\alpha}}.
\end{align}
The scalar $\alpha$ is called the \textit{stability index} of $\bm{Z}$, while
$\nu$ will be referred as the \textit{angular measure} of $\bm{Z}$. Since the
distribution of a simple max-stable random vector is characterized equivalently
by $\mu$ alone or $\alpha$ and $\nu$, we will parametrize it with either of
them. We denote this by $\bm{Z} \sim \mathcal{MS}(\mu)$ and
$\bm{Z} \sim \mathcal{MS}(\alpha, \nu)$ respectively.

Max-infinitely divisible distributions generalize the concept of max-stable
random variables: the distribution of a random vector $\bm{Z}$ is said to be
\textit{max-infinitely divisible} (max-id) if $F_{Z}(\bm{x})^{t}$ is a c.d.f.
for any positive power $t$, where
$F_{\bm{Z}}(\bm{x}) = \prob(\bm{Z} \leq \bm{x})$. This is equivalent to asking
that for every $n \in \N^{*}$, there exist $n$ \iid random vectors
$\bm{Z}_{n,1},\dots,\bm{Z}_{n,n}$ such that
\[
  \bm{Z} \eqdis \bigoplus_{i=1}^{n}\bm{Z}_{n,i}.
\]
In dimension $1$, any probability distribution is max-infinitely divisible, but
this is not true in higher dimension. Identity \eqref{prelim_LePage_1} still
holds for max-id distributions. More precisely, a random vector $\bm{Z}$ is
max-id if and only if there exists $\bm{\ell} \in [-\infty, +\infty)^{d}$ such
that
\[
  \bm{Z} \eqdis \bigoplus_{i = 1}^{\infty}\bm{y}_{i}.
\]
where $(\bm{y}_{i})_{i\geq 1}$ is a Poisson process on $\El$ whose intensity
measure $\mu$ satisfies
\[
  \mu[-\bm{\infty}, \bm{x}]^{c} = -\log F_{\bm{Z}}(\bm{x}),\ \bm{x} \in \El.
\]
In the special case of a simple max-stable random vector, because of
\eqref{prelim_Fréchet}, $\mu[\bm{0}, \bm{x}]^{c}$ is infinite as soon as any of
the coordinates of $\bm{x}$ is null. The converse is true:

\begin{lemma}
  Let $\mu$ be the exponent measure of a simple max-stable random vector. Then
  $\mu[\bm{0}, \bm{x}]^{c}$ is finite if and only if $\bm{x} \in \Eos$.
\end{lemma}
\begin{proof}
  The direct implication has already been proved. To get the reverse statement,
  assume $\alpha = 1$ for simplicity. Since $\bm{x} \in \Eos$, the scalar
  $\min(x^{1},\dots,x^{d})$ is positive and
  \begin{align*}
    \mu[\bm{0},\bm{x}]^{c} & = \int_{\S_{+}^{d-1}}\int_{0}^{\infty}\ind_{\lb r\bm{u} \nleq \bm{x}\rb} \frac{1}{r^{2}}\dif r \dif\nu(\bm{u})             \\
                           & = \int_{\S_{+}^{d-1}}\int_{0}^{\infty}\ind_{\bigcup_{j=1}^{d}\lb ru^{j} > x^{j} \rb} \frac{1}{r^{2}}\dif r \dif\nu(\bm{u}) \\
                           & = \int_{\S_{+}^{d-1}}\int_{\min \frac{\bm{x}}{\bm{u}}}^{\infty} \frac{1}{r^{2}}\dif r \dif\nu(\bm{u})                      \\
                           & = \int_{\S_{+}^{d-1}} \max \frac{\bm{u}}{\bm{x}}\dif\nu(\bm{u}) \leq \frac{1}{\min \bm{x}}.
  \end{align*}
\end{proof}
Another useful lemma regarding the exponent measure of a max-stable random
vector is the following:

\begin{lemma}\label{prelim_esp_mu}
  Let $\bm{Z} \sim \mathcal{MS}(\mu)$ and $k \in \N$. Let
  $\bm{1} \coloneq (1,\dots,1) \in \Eos$. Then one has:
  \[
    \esp\big[\big(\mu[\bm{0},\bm{Z}]^{c}\big)^{k}\big] \leq d\big(\mu[\bm{0},\bm{1}]^{c}\big)^{k}k!.
  \]
  Furthermore, define $\log \bm{x}$ as $(\log x^{1},\dots, \log x^{j})$ for
  $\bm{x} \in \Eos$. Then
  \[
    \esp\big[\Vert \log \bm{Z} \Vert_{1}^{k}\big] < +\infty,
  \]
  where $\Vert \bm{x} \Vert_{1} = \sum_{j=1}^{d}\vert x^{j} \vert$. Finally, the
  following is true:
  \[
    \int_{\Eo} \esp\big[\Vert \log (\bm{Z}\oplus \bm{y}) \Vert_{1}^{k} \ind_{\lb \bm{y} \nleq \bm{Z}\rb}\big]\dif\mu(\bm{y}) < +\infty.
  \]
\end{lemma}

\begin{proof}
  Each marginal $Z^{j}$ of $\bm{Z}$ has the Fréchet distribution
  $\mathcal{F}(1)$, so that $1/Z^{j}$ has the exponential distribution
  $\mathcal{E}(1)$ and $\esp[(Z^{j})^{-k}] = k!$. Observe that we have:
  \[ [\bm{0}, \bm{x}]^{c} \subseteq [\bm{0}, (\min \bm{x})\bm{1}]^{c}.
  \]
  The homogeneity property of $\mu$ then yields
  \begin{align*}
    \esp\big[\big(\mu[\bm{0},\bm{Z}]^{c}\big)^{k}\big] & \leq \esp\big[\big(\mu[\bm{0},(\min \bm{Z})\bm{1}]^{c}\big)^{k}\big]   \\
                                                       & = \big(\mu[\bm{0},\bm{1}]^{c}\big)^{k}\esp\big[(\min \bm{Z})^{-k}\big] \\
                                                       & = \big(\mu[\bm{0},\bm{1}]^{c}\big)^{k}\esp[\max \bm{Z}^{-k}]           \\
                                                       & \leq \big(\mu[\bm{0},\bm{1}]^{c}\big)^{k}d\esp[(Z^{j})^{-k}],
  \end{align*}
  by bounding $\max 1/\bm{Z}$ by $\sum_{j=1}^{d}1/Z^{j}$, since all the $Z^{j}$
  are positive. The second statement is a direct consequence of the fact that if
  $Z \sim \mathcal{F}(1)$, then $\log Z$ has the Gumbel distribution. Its
  non-negative moments are thus all finite. We will get that the last
  expectation is finite if we can prove that
  \[
    \int_{\Eo} \esp\big[\vert \log (Z^{j}\oplus y^{j}) \vert^{k} \ind_{\lb \bm{y} \nleq \bm{Z}\rb}\big]\dif\mu(\bm{y}) < +\infty
  \]
  for every $j \in \lbra 1,d \rbra$. A simple case distinction yields that
  result:
  \begin{align*}
    \esp\big[\vert \log (Z^{j}\oplus & y^{j}) \vert^{k} \ind_{\lb \bm{y} \nleq \bm{Z}\rb}\big]                                                                                                                                                                                        \\
                                     & = \esp\big[\vert \log (Z^{j}\oplus y^{j}) \vert^{k} \ind_{\lb \bm{y} \nleq \bm{Z}\rb}\ind_{\lb Z^{j} > y^{j}\rb}\big] + \esp\big[\vert \log (Z^{j}\oplus y^{j}) \vert^{k} \ind_{\lb \bm{y} \nleq \bm{Z}\rb}\ind_{\lb Z^{j} \leq y^{j}\rb}\big] \\
                                     & = \esp\big[\vert \log Z^{j} \vert^{k} \ind_{\lb \bm{y} \nleq \bm{Z}\rb}\ind_{\lb Z^{j} > y^{j}\rb}\big] + \vert \log y^{j} \vert^{k}\prob(Z^{j} \leq y^{j})                                                                                    \\
                                     & \leq \esp\big[\vert \log Z^{j} \vert^{k} \ind_{\lb \bm{y} \nleq \bm{Z}\rb}\big] + \vert \log y^{j} \vert^{k}e^{-\frac{1}{y^{j}}}.
  \end{align*}
  The second term is integrable with respect to $\bm{y}$ on $\Eo$, as one can
  see by using the polar decomposition:
  \[
    \int_{\Eo}\vert \log y^{j} \vert^{k}e^{-\frac{1}{y^{j}}} \dif\mu(\bm{y}) = \int_{\Epol}\vert \log ru^{j} \vert^{k}e^{-\frac{1}{ru^{j}}}\frac{1}{r^{2}}\dif r \dif\nu(\bm{u}).
  \]
  As for the first part, a Fubini argument gives
  \begin{align*}
    \int_{\Eo}\esp\big[\vert \log Z^{j} \vert^{k} \ind_{\lb \bm{y} \nleq \bm{Z}\rb}\big]  \dif\mu(\bm{y}) = \esp\big[\vert \log Z^{j} \vert^{k} \mu[\bm{0}, \bm{Z}]^{c}\big].
  \end{align*}
  One then concludes using Cauchy-Schwarz inequality and the first two points of
  this lemma.
\end{proof}

\subsection{Stochastic extremal integrals}\label{prelim_SEI}

We give a short account of the notion of stochastic extremal integral. The
interested reader is referred to \cite{Stoev05} for a much more thorough
presentation. We focus of the properties of the stochastic extremal integral we
will need of later. In this subsection, we will denote the scale parameter
$\sigma$ of a Fréchet random variable $Z$ by
$\Vert Z \Vert_{\alpha} \coloneq \sigma$.

\begin{definition}

  Let $(E, \mathcal{E},\mu)$ be a measure space,
  $\mathcal{E}_{0} \coloneq \lb A\in \mathcal{E},\ \mu(A) < \infty \rb$ and
  $\mathbf{L}^{0}(\Omega)$ the set of real random variables on $\Omega$. Let
  $\alpha > 0$. We say that a function
  $M_{\alpha} : \mathcal{E}_{0} \to \mathbf{L}^{0}(\Omega)$ is a \textit{random
    sup-measure with control measure $\mu$} if it satisfies the three following
  conditions:

  \begin{enumerate}

    \item (independently scattered) For any collection of disjoint sets
          $(A_{j})_{1\leq j \leq n}$ in $\mathcal{E}_{0}$, the random variables
          $(M_{\alpha}(A_{j}))_{1\leq j \leq n}$ are independent.

    \item ($\alpha$-Fréchet) For any
          $A\in \mathcal{E}_{0},\ M_{\alpha} \sim \mathcal{F}\big(\alpha, \mu(A)^{1/\alpha})\big)$.

    \item ($\sigma$-sup-additive) For any collection of disjoint sets
          $(A_{i})_{i\geq 1}$ in $\mathcal{E}_{0}$ such that
          $\bigcup_{i}A_{i} \in \mathcal{E}_{0}$, we have:
          \[
            M_{\alpha}\Big(\biguplus_{i = 1}^{\infty}A_{i}\Big) = \bigoplus_{i=1}^{\infty}M_{\alpha}(A_{i})\ \mathrm{a.s.}.
          \]
  \end{enumerate}
\end{definition}

Using this definition, we introduce the extremal integral for simple functions.

\begin{definition}
  Let $f$ be a simple function on $E$:
  \[
    f(x) = \sum_{i=1}^{n}a_{i}\ind_{A_{i}}(x),\ x \in E
  \]
  where $a_{i}$ are non-negative numbers and the $A_{i}$ are disjoint. The
  \textit{extremal integral} of $f$ with respect to the random sup-measure
  $M_{\alpha}$ is defined as:
  \[
    \mxint{E}{}f(x)\dif M_{\alpha}(x) \coloneq \bigoplus_{i=1}^{n}a_{i}M_{\alpha}(A_{i}).
  \]

\end{definition}

This definition can be extended to more general integrands, as explained in
\cite{Stoev05}. Now we list the most important properties of this integral for
our purposes:

\begin{theorem}

  Let $f$ be a non-negative function such that
  $\int_{E}f(x)^{\alpha}\dif \mu(x)$ is finite. Then the extremal integral of
  $f$ on $E$ exists and is a random variable $Z$ with Fréchet distribution
  $\mathcal{F}(\alpha, \Vert Z \Vert_{\alpha})$, where
  \[
    \Vert Z \Vert_{\alpha}^{\alpha} = \Big\Vert \mxint{E}{} f(x)\dif M_{\alpha}(x) \Big\Vert_{\alpha}^{\alpha} = \int_{E}f(x)^{\alpha}\dif \mu(x).
  \]
\end{theorem}

This theorem can be roughly seen as an 'extremal' counterpart of the It{\=o}
isometry, although $\alpha$-Fréchet random variables never belong to
$\mathbf{L}^{\alpha}(\R^{*}_{+})$. In the sequel we will use the notation
\[
  \mathbf{L}^{\alpha}_{+}(E, \mu) \coloneq \Big\lb f : E \to \R_{+},\ \int_{E}f(x)^{\alpha}\dif \mu(x) < +\infty \Big\rb.
\]

\begin{proposition}

  The stochastic extremal integral satisfies the following properties.

  \begin{enumerate}

    \item (Max-linearity) For all $f,g \in \mathbf{L}_{+}^{\alpha}(E, \mu)$, we
          have
          \[
            \Big(\lambda \mxint{E}{} f\dif M_{\alpha}\Big) \oplus \Big(\mu \mxint{E}{} g\dif M_{\alpha}\Big) = \mxint{E}{} \big(\lambda f \oplus \mu g\big) \dif M_{\alpha},\ \lambda, \mu \geq 0.
          \]

    \item (Independence) The extremal integrals of $f$ and $g$ are independent
          if and only if $f$ and $g$ have disjoint supports, that is:
          \[
            \mxint{E}{} f\dif M_{\alpha}\ \mathrm{and\ } \mxint{E}{} g\dif M_{\alpha}\ \mathrm{are\ independent\ if\ and\ only\ if\ } fg = 0\ \mu\mathrm{-a.e.}
          \]

    \item (Monotonicity)
          \[
            f\leq g\ \mu\mathrm{-a.e.\ if\ and\ only\ if\ } \mxint{E}{} f\dif M_{\alpha} \leq \mxint{E}{} g\dif M_{\alpha}\ \mu\mathrm{-a.e.}
          \]
          In particular, $f = g\ \mu$-a.e. if and only if the associated
          extremal integrals are equal $\mu$-a.e.

  \end{enumerate}
\end{proposition}

\section{Stochastic analysis for max-id distributions}
\label{sec:stoch-analys-max}
Let $\mathfrak{N}_{\El}$ be the space of configurations of $\El$. If $\eta$ is a
Poisson process on $\El$, we will denote by:
\[
  \L^{2}(\mathfrak{N}_{\El}, \prob_{\eta}) \coloneq \L^{2}(\prob_{\eta})
\]
the set of $\prob_{\eta}$-square-integrable functions from $\mathfrak{N}_{\El}$
to $\R$. Let $\m : \mathfrak{N}_{\El} \to \El$ be the coordinate-wise maximum
over $\mathfrak{N}_{\El}$:
\[
  \m(\phi) \coloneq \bigoplus_{\bm{y} \in \phi}\bm{y},
\]
where $\phi$ is some configuration on $\El$.

The de Haan-LePage representation \eqref{prelim_LePage_1} implies that any
functional $f(\bm{Z})$ of a max-id random vector $\bm{Z}$ can be realized as a
functional $(f\circ \m)(\eta) = \bar{f}(\eta)$ of some underlying Poisson
process $\eta = (\bm{y}_{i})_{i\geq 1}$ on $\El$, for some
$\bm{\ell} \in [-\infty, +\infty)^{d}$:
\[ \begin{tikzcd}[row sep=3.9em, column sep=5em] \mathfrak{N}_{\El} \arrow{r}{\m} \arrow[swap]{dr}{\bar{f}} & \El \arrow{d}{f}
    \\%
    & \R
  \end{tikzcd}
\]
The application $\m$ satisfies the elementary but important relation:
\begin{align}\label{3_cup_oplus}
  \m(\phi + \delta_{\bm{y}}) = \m(\phi)\oplus \bm{y}.
\end{align}
for every $\bm{y} \in \El$ and configuration $\phi \in \mathfrak{N}_{\El}$, with
$\delta_{\bm{y}}$ the Dirac measure at $\bm{y}$.

The homogeneity property of the exponent measure of max-stable random vectors
provides us with the following expansion:

\begin{proposition}
  Let $\bm{x} \in \R_{+}^{d}$, $\sigma > 0$ and $\bm{Z}$ a random vector of
  $\R_{+}^{d}$. If $\bm{Z} \sim \mathcal{MS}(1, \nu)$ and $\mu$ is its
  associated exponent measure, then for all non-negative $f : \Eos \to \R_{+}$
  and $\bm{x} \in \Eos$:
  \begin{multline}
\label{3_chaos_MS}
\esp\big[f(\bm{x}  \oplus \sigma\bm{Z})\big]    \\
\begin{aligned}
    &= F_{\sigma\bm{Z}}(\bm{x})f(\bm{x}) + F_{\sigma\bm{Z}}(\bm{x})\sum_{n=1}^{\infty}\frac{\sigma^{n}}{n!}\int_{([\bm{0}, \bm{x}]^{c})^{n}}f\big(\bm{x}\oplus \bm{y}_{1} \oplus \dots \oplus \bm{y}_{n}\big)\prod_{i=1}^{n}\dif \mu(\bm{y}_{i})\\
                      & = e^{-\sigma\mu[\bm{0}, \bm{x}]^{c}}f(\bm{x}) + e^{-\sigma\mu[\bm{0}, \bm{x}]^{c}}\sum_{n=1}^{\infty}\frac{\sigma^{n}}{n!}\int_{([\bm{0}, \bm{x}]^{c})^{n}}f\big(\bm{x}\oplus \bm{y}_{1} \oplus \dots \oplus \bm{y}_{n}\big)\prod_{i=1}^{n}\dif \mu(\bm{y}_{i}). 
    \end{aligned}
\end{multline}
  Conversely, if equality \eqref{3_chaos_MS} holds for all non-negative
  $f : \Eos \to \R_{+}$, then $\bm{Z} \sim \mathcal{MS}(1, \nu)$.
\end{proposition}

\begin{proof}
 The second equality is an easy consequence of the homogeneity property
  \eqref{3_chaos_MS} of the exponent measure $\mu$:
  \[
    F_{\sigma\bm{Z}}(\bm{x}) = F_{\bm{Z}}(\sigma^{-1}\bm{x}) = e^{-\mu[\bm{0}, \sigma^{-1}\bm{x}]^{c}} = e^{-\sigma\mu[\bm{0}, \bm{x}]^{c}}.
  \]
  Thanks to the fundamental equality in law \eqref{prelim_LePage_1}, it stands
  true that
  \[
    \bm{x}\oplus \sigma\bm{Z} \eqdis \bm{x}\oplus \sigma\bigoplus_{i=1}^{\infty}\bm{y}_{i} = \m(\bm{x}\oplus \sigma\eta).
  \]
  Recall the following result, which can be found for instance in
  \cite{Privault09}: a random measure $\eta$ with finite intensity measure $\pi$
  on a subset $E$ of $\R^{d}$ is a Poisson process if and only if
  \begin{align}\label{3_chaos_PPP}
    \esp[\bar{f}(\eta)] & = e^{-\pi(E)}\bar{f}(\emptyset) + e^{-\pi(E)}\sum_{n=1}^{\infty}\frac{1}{n!}\int_{E^{n}}\bar{f}(\lb \bm{y}_{1},\dots,\bm{y}_{n}\rb)\dif\pi^{n}(\bm{y}_{1},\dots,\bm{y}_{n}),
  \end{align}
  for all non-negative $\bar{f} : \mathfrak{N}_{E} \to \R_{+}$. We cannot apply
  identity \eqref{3_chaos_PPP} immediately, as the Poisson process
  $\eta = (\bm{y}_{i})_{i\geq 1}$ does not have finite intensity on $\Eo$. But
  thanks to the homogeneity property \eqref{prelim_homo} of $\rho_{1}$, we have
  \[
    \bm{x}\oplus \sigma\bigoplus_{i=1}^{\infty}\bm{y}_{i} \eqdis \bm{x}\oplus\bigoplus_{i=1}^{\infty}\bm{w}_{i},
  \]
  where $(\bm{w}_{i})_{i\geq 1}$ is a Poisson point process on
  $[\bm{0}, \bm{x}]^{c}$, with intensity measure
  $\sigma\mu(\cdot \cap [\bm{0}, \bm{x}]^{c})$. Since $\bm{x}$ belongs to
  $\Eos$, we know that $\mu[\bm{0}, \bm{x}]^{c}$ is finite. We deduce the
  announced result by applying \eqref{3_chaos_PPP} to
  $\pi = \sigma\mu(\cdot \cap [\bm{0}, \bm{x}]^{c})$.

  To prove the reverse implication, fix $\bm{x} \in \Eos$, take
  $f = \ind_{[\bm{0}, \sigma\bm{x}]}$ and evaluate identity \eqref{3_chaos_MS}
  at $\sigma\bm{x}$, so that
  \[
    F_{\bm{Z}}(\bm{x}) = e^{-\sigma\mu[\bm{0}, \sigma\bm{x}]^{c}}\ind_{[\bm{0}, \sigma\bm{x}]}(\sigma\bm{x}) = e^{-\mu[\bm{0}, \bm{x}]^{c}}
  \]
  thanks to the homogeneity property \eqref{prelim_homo} of the exponent measure
  $\mu$. The right-hand side is the c.d.f. of the max-stable distribution
  $\mathcal{MS}(1, \nu)$.
\end{proof}


Following \cite{Last17}, we define the \textit{discrete gradient} on
$\mathfrak{N}_{\El}$by:
\[
  D_{\bm{y}}\bar{f}(\phi) \coloneq \bar{f}(\phi + \delta_{\bm{y}}) - \bar{f}(\phi),
\]
with $\bar{f} : \mathfrak{N}_{\El} \to \El$ and $\bm{y} \in \El$. Next, set
\[
  D^{\oplus}_{\bm{y}}f(\bm{x}) \coloneq f(\bm{x}\oplus \bm{y}) - f(\bm{x}).
\]
If $\phi$ is a configuration on $\El$ and $\bm{x} = \m(\phi)$, the two previous
definitions coincide, as we get from \eqref{3_cup_oplus} that:
\[
  D^{\oplus}_{\bm{y}}f(\bm{x}) = D_{\bm{y}}\bar{f}(\phi),
\]
where $\bar{f} = f\circ \m$. More generally, we denote by
$D^{\oplus}_{\bm{y}_{1},\dots,\bm{y}_{n}}$ the composition
$D^{\oplus}_{\bm{y}_{1}} \circ \dots \circ D^{\oplus}_{\bm{y}_{n}}$.

As a consequence of the chaos decomposition on the Poisson space, we have a
covariance identity for max-id random vectors:

\begin{proposition}
  Let $\bm{Z}$ be a max-id random vector with exponent measure $\mu$ on $\Els$,
  for some $\ell \in [-\bm{\infty}, +\bm{\infty})$. Let
  $f,g \in \mathbf{L}^{2}(\prob_{\bm{Z}})$. Set:
  \[
    T^{\oplus}_{n}f(\bm{x}_{1},\dots,\bm{x}_{n}) \coloneq \esp\big[D^{\oplus}_{\bm{x}_{1},\dots, \bm{x}_{n}}f(\bm{Z})\big]
  \]
  and for all $u, v \in \mathbf{L}^{2}(\El^{n}, \mu^{\otimes n})$:
  \[
    \langle u,v \rangle_{n} \coloneq \int_{\El^{n}}u(\bm{x}_{1},\dots,\bm{x}_{n})v(\bm{x}_{1},\dots,\bm{x}_{n})\dif^{n}\mu(\bm{x}_{1},\dots,\bm{x}_{n}).
  \]
  We have the following identity:
  \begin{align}\label{3_cov_id_max1}
    \Cov\big(f(\bm{Z}), g(\bm{Z})\big) = \sum_{n=1}^{\infty}\frac{1}{n!}\langle T^{\oplus}_{n}f, T^{\oplus}_{n}g \rangle_{n}.
  \end{align}
\end{proposition}

\begin{proof}
  Recall the Fock space representation for Poisson processes with
  $\sigma$-finite intesity measures:
  \begin{align}\label{3_Fock}
    \Cov\big(\bar{f}(\eta), \bar{g}(\eta)\big) = \sum_{n=1}^{\infty} \frac{1}{n!}\langle T_{n}\bar{f}, T_{n}\bar{g} \rangle_{n},
  \end{align}
  where
  \[
    T_{n}f(\bm{x}_{1},\dots,\bm{x}_{n}) \coloneq \esp\big[D_{x_{1},\dots,x_{n}}f(\eta)\big].
  \]
  Applying this identity to the Poisson process $\eta$ with intensity measure
  $\mu$ while taking $\bar{f} = f\circ \m$ and $\bar{g} = g\circ \m$. Clearly
  $\bar{f}$ and $\bar{g}$ belong to $\mathbf{L}^{2}(\prob_{\eta})$. Finally, we
  recognize the terms inside the series in \eqref{3_Fock} by using identity
  \eqref{prelim_LePage_1}.
\end{proof}

In dimension $1$, it is possible to greatly simplify the previous covariance
identity thanks to the next lemma:

\begin{lemma}
  Let $f : \R \to \R$ and $x, r_{1},\dots,r_{n} \in \R$ for some $n \geq 1$. Set
  $r_{(n)} \coloneq \min(r_{1},\dots, r_{n})$. We have
  \begin{align}\label{3_grad_max_itrted}
    D^{\oplus}_{r_{1},\dots,r_{n}}f(x) = (-1)^{n-1}D^{\oplus}_{x \odot r_{(n)}}f(x) =
    \begin{cases}
      (-1)^{n-1}D^{\oplus}_{r_{(n)}}f(x) & \text{if\ } x\leq r_{(n)} \\
      0                                  & \text{otherwise}.
    \end{cases}
  \end{align}
\end{lemma}
\begin{proof}
  We prove this result by induction on $n$. The case $n=1$ is trivial. Assume
  the proposition holds true for some $n$. Then we have:
  \begin{align*}
    D^{\oplus}_{r_{1},\dots,r_{n}, r_{n+1}}f(x) & = D^{\oplus}_{r_{n+1}}D^{\oplus}_{r_{1},\dots,r_{n}}f(x)                                                            \\
                                                & = (-1)^{n-1}D^{\oplus}_{r_{n+1}}D^{\oplus}_{x \odot r_{(1)}}f(x)                                                    \\
                                                & = (-1)^{n-1}\Big[f\big(x \oplus r_{(n)} \oplus r_{n+1}\big) - f(x \oplus r_{(n)}) - f(x \oplus r_{n+1}) + f(x)\Big] \\
                                                & = (-1)^{n}D^{\oplus}_{x \odot r_{(n+1)}}f(x).
  \end{align*}
  because
  \begin{align*}
    D^{\oplus}_{r_{n+1}}D^{\oplus}_{x \odot r_{(1)}}f(x) = f\big(x \oplus (x \odot r_{(n)}) \oplus r_{n+1}\big) - f\big(x \oplus (x \odot r_{(n)})\big) - f(x \oplus r_{n+1}) + f(x).
  \end{align*}
  By distinguishing cases, depending on the rank of $r_{n+1}$ with respect to to
  $x$ and $r_{(n)}$, we find that this is equal to
  $-D^{\oplus}_{x \odot r_{(n+1)}}f(x)$.
\end{proof}

\begin{proposition}
  Let $Z$ be a max-id random variable on $(\ell, \infty)$. Let
  $f,g \in \mathbf{L}^{2}(\prob_{Z})$. Then
  \begin{align}\label{3_cov_id_max2}
    \Cov\big(f(Z), g(Z)\big) = \int_{\ell}^{\infty}\esp\big[D^{\oplus}_{r}f(Z)\big]\esp\big[D^{\oplus}_{r}g(Z)\big]\frac{\dif\mu(r)}{F_{Z}(r)}\cdotp
  \end{align}
\end{proposition}

\begin{proof}
  Thanks to \eqref{3_grad_max_itrted}, we have
  \begin{align*}
    T^{\oplus}_{n}f(r_{1},\dots,r_{n})T^{\oplus}_{n}g(r_{1},\dots,r_{n}) & = \esp\big[D^{\oplus}_{r_{1},\dots, r_{n}}f(Z)\big]\esp\big[D^{\oplus}_{r_{1},\dots, r_{n}}g(Z)\big] \\
                                                                         & = \esp\big[D^{\oplus}_{r_{(n)}}f(Z)\big]\esp\big[D^{\oplus}_{r_{(n)}}g(Z)\big].
  \end{align*}
  Hence:
  \begin{align*}
    \langle T^{\oplus}_{n}f, T^{\oplus}_{n}g \rangle_{n} & = \sum_{i=1}^{n}\int_{E_{\ell}^{n}}\esp\big[D^{\oplus}_{r_{i}}f(Z)\big]\esp\big[D^{\oplus}_{r_{i}}g(Z)\big]\ind_{\lb r_{(n)} = r_{i}\rb}\dif \mu(r_{1})\dots \dif\mu(r_{n})                       \\
                                                         & = n\int_{E_{\ell}^{n}}\esp\big[D^{\oplus}_{r_{n}}f(Z)\big]\esp\big[D^{\oplus}_{r_{n}}g(Z)\big]\ind_{\lb r_{(n) = r_{n}}\rb}\dif \mu(r_{1})\dots \dif\mu(r_{n})                                    \\
                                                         & = n\int_{E_{\ell}}\esp\big[D^{\oplus}_{r_{n}}f(Z)\big]\esp\big[D^{\oplus}_{r_{n}}g(Z)\big]\int_{r_{n}}^{\infty} \dots \int_{r_{n}}^{\infty} 1 \dif \mu(r_{1})\dots \dif\mu(r_{n-1})\dif\mu(r_{n}) \\
                                                         & = n\int_{E_{\ell}}\esp\big[D^{\oplus}_{r_{n}}f(Z)\big]\esp\big[D^{\oplus}_{r_{n}}g(Z)\big]\big(-\log F(r_{n})\big)^{n-1}\dif\mu(r_{n}),
  \end{align*}
  since $\mu[0, x]^{c} = -\log F(x)$. Finally
  \begin{align*}
    \Cov\big(f(Z), g(Z)\big) & = \sum_{n=1}^{\infty}\frac{1}{n!}\langle T^{\oplus}_{n}f, T^{\oplus}_{n}g \rangle_{n}                                                              \\
                             & = \sum_{n=0}^{\infty}\frac{1}{n!}\int_{E_{\ell}}\esp\big[D^{\oplus}_{r}f(Z)\big]\esp\big[D^{\oplus}_{r}g(Z)\big]\big(-\log F(r)\big)^{n}\dif\mu(r) \\
                             & = \int_{E_{\ell}}\esp\big[D^{\oplus}_{r}f(Z)\big]\esp\big[D^{\oplus}_{r}g(Z)\big]e^{-\log F(r)}\dif\mu(r).
  \end{align*}
\end{proof}

It seems that this simplification breaks in higher dimension. Nonetheless, the
Poincaré inequality still holds.

\begin{proposition}[Max-id Poincaré inequality]
  Let $\bm{Z}$ be a max-id random vector with exponent measure $\mu$ supported
  by $\El$ for some $\bm{\ell} \in [-\bm{\infty}, +\bm{\infty})$, and assume
  $f \in \mathbf{L}^{2}(\prob_{\bm{Z}})$. We have:
  \begin{align}\label{3_Poincaré_max-id}
    \var\big(f(\bm{Z})\big) \leq \int_{\El}\esp\big[\big(f(\bm{Z}\oplus \bm{x}) - f(\bm{Z})\big)^{2}\big]\dif\mu(\bm{x}).
  \end{align}
\end{proposition}

\begin{proof}
  The well-known Poincaré inequality for Poisson processes states that if $\eta$
  is a Poisson process on some measurable space $E$, with $\sigma$-finite
  intensity measure $\mu$, then one has
  \[
    \var\big(\bar{f}(\eta)\big) \leq \int_{E}\esp\big[\big(\bar{f}(\eta + \delta_{\bm{x}}) - f(\eta)\big)^{2}\big]\dif\mu(\bm{x}),
  \]
  for any $\bar{f} \in \mathbf{L}^{2}(\prob_{\eta})$. We apply this result to
  $\bar{f} \coloneq f \circ \m \in \mathbf{L}^{2}(\prob_{\eta})$, yielding:
  \begin{align*}
    \var\big(f(\bm{Z})\big) = \var\big(\bar{f}(\eta)\big) & \leq \int_{\El}\esp\big[\big(\bar{f}(\eta + \delta_{\bm{x}}) - \bar{f}(\eta)\big)^{2}\big]\dif\mu(\bm{x}) \\
                                                          & = \int_{\El}\esp\big[\big(f(\bm{Z}\oplus \bm{x}) - f(\bm{Z})\big)^{2}\big]\dif\mu(\bm{x}),
  \end{align*}
  thanks to identity \eqref{3_cup_oplus}. Alternatively, one could have proved
  this result by using the covariance identity \eqref{3_cov_id_max1}, the same
  way the original Poincaré inequality is demonstrated in \cite{Last17} (page
  193).
\end{proof}

What we have proved is sometimes called a \textit{first-order Poincaré
  inequality}. We now prove a \textit{second-order Poincaré inequality} for
max-id random vectors.

\begin{proposition}[Max-id second-order Poincaré inequality]
  Let $N \sim \mathcal{N}(0,1)$. Let $\bm{Z}$ be a max-id random vector with
  exponent measure $\mu$ on $\El$ for some
  $\bm{\ell} \in [-\bm{\infty}, +\bm{\infty})$, and $f : \Els \to \R$. Assume
  that
  \[
    \esp[f(\bm{Z})] = 0\ \text{and\ } \var\big(f(\bm{Z})\big) = 1.
  \]
  Then
  \[
    \mathrm{d}_{W}\big(f(\bm{Z}), N\big) \leq \gamma_{1} + \gamma_{2} + \gamma_{3},
  \]
  with
  \begin{equation*}
    \begin{split}
       & \gamma_{1} \coloneq 2\Big(\int_{\El^{3}}\big(\esp\big[(D^{\oplus}_{\bm{x}}f(\bm{Z}))^{2}(D^{\oplus}_{\bm{y}}f(\bm{Z}))^{2}\big]\big)^{\frac{1}{2}}\big(\esp\big[(D^{\oplus}_{\bm{x}, \bm{z}}f(\bm{Z}))^{2}(D_{\bm{y}, \bm{z}}f(\bm{Z}))^{2}\big]\big)^{\frac{1}{2}}\dif\mu^{3}(\bm{x}, \bm{y}, \bm{z})\Big)^{\frac{1}{2}} \\
       & \gamma_{2} \coloneq \Big(\int_{\El^{3}}\esp\big[(D^{\oplus}_{\bm{x}, \bm{z}}f(\bm{Z}))^{2}(D^{\oplus}_{\bm{y}, \bm{z}}f(\bm{Z}))^{2}\big]\dif\mu(\bm{x})\dif\mu^{2}(\bm{y}, \bm{z})\Big)^{\frac{1}{2}}                                                                                                                    \\
       & \gamma_{3} \coloneq \int_{\El}\esp\big[\vert D^{\oplus}_{\bm{x}}f(\bm{Z})\vert^{3}\big]\dif\mu(\bm{x}).
    \end{split}
  \end{equation*}
\end{proposition}

\begin{proof}
  Recall the following theorem from \cite{Last16}: Let $\eta$ be a Poisson
  process over $\El$ with intensity measure $\lambda$. Denote by $\bar{f}(\eta)$
  a Poisson functional and assume that it is centered and has unit variance.
  Set:
  \begin{equation*}
    \begin{split}
       & \gamma_{1} \coloneq 2\Big(\int_{\El^{3}}\big(\esp\big[(D_{\bm{x}}\bar{f}(\eta))^{2}(D_{\bm{y}}\bar{f}(\eta))^{2}\big]\big)^{\frac{1}{2}}\big(\esp\big[(D_{\bm{x}, \bm{z}}\bar{f}(\eta))^{2}(D_{\bm{y}, \bm{z}}\bar{f}(\eta))^{2}\big]\big)^{\frac{1}{2}}\dif\lambda^{3}(\bm{x}, \bm{y}, \bm{z})\Big)^{\frac{1}{2}} \\
       & \gamma_{2} \coloneq \Big(\int_{\El^{3}}\esp\big[(D_{\bm{x}, \bm{z}}\bar{f}(\eta))^{2}(D_{\bm{y}, \bm{z}}\bar{f}(\eta))^{2}\big]\dif\lambda(\bm{x})\dif\lambda^{2}(\bm{y}, \bm{z})\Big)^{\frac{1}{2}}                                                                                                               \\
       & \gamma_{3} \coloneq \int_{\El}\esp\big[\vert D_{\bm{x}}\bar{f}(\eta)\vert^{3}\big]\dif\lambda(\bm{x}).
    \end{split}
  \end{equation*}
  If $\lambda$ is $\sigma$-finite, then:
  \[
    \mathrm{d}_{W}\big(\bar{f}(\eta), N\big) \leq \gamma_{1} + \gamma_{2} + \gamma_{3},
  \]
  where $N$ is a random variable with standard Gaussian distribution. We use
  this result for $\bar{f} = f \circ \m$ and $\lambda = \mu$
\end{proof}

The last result of this section is a modified logarithmic Sobolev inequality for
max-id distributions. The original result on the Poisson space has been found
and proved by Wu in \cite{Wu00}.

\begin{proposition}[Max-id modified logarithmic Sobolev inequality]
  Let $\bm{Z}$ be a max-id random vector on $\Els$, with
  $\bm{\ell} \in [-\bm{\infty}, +\bm{\infty})$, and exponent measure $\mu$ on
  $\El$. Set $\Phi(x) \coloneq x\log x$ for $x>0$ and
  \[
    \mathrm{Ent}_{\bm{Z}}(f) \coloneq \esp\big[(\Phi \circ f)(\bm{Z})\big] - \Phi\big(\esp\big[f(\bm{Z})\big]\big)
  \]
  for every $\prob_{\bm{Z}}$-almost surely positive, $\prob_{\bm{Z}}$-integrable
  $f$. Then one has:
  \[
    \mathrm{Ent}_{\bm{Z}}(f) \leq \int_{\El}\esp\big[D^{\oplus}_{\bm{y}}(\Phi \circ f)(\bm{Z}) - (\Phi' \circ f)(\bm{Z})D^{\oplus}_{\bm{y}}f(\bm{Z})\big]\dif\mu(\bm{y}).
  \]
\end{proposition}

\begin{proof}
  Apply the modified logarithmic Sobolev inequality for Poisson processes stated
  in \cite{Wu00} to the function $\bar{f} = f \circ \m$.
\end{proof}


\section{The max-stable Ornstein-Uhlenbeck operator}
\label{sec:max-stable-ornstein}
Let $\pms_{\alpha,\nu}$ denote the probability distribution of a max-stable
random vector $\bm{Z} \sim \mathcal{MS}(\alpha, \nu)$ and set:
\[
  \L^{p}(\pms_{\alpha, \nu}) \coloneq \L^{p}(\Eos, \pms_{\alpha,\nu}),\ p\in [1, +\infty].
\]

%

\subsection{The case $\alpha = 1$}

Fix a reference norm $\Vert\cdot\Vert$ on $\R^{d}$. Recall that $\nu$ is a
finite measure on $\S_{+}^{d-1}$ satisfying the moment constraints relation
\eqref{prelim_moment_constraints}, and $\mu$ denotes the exponent measure of a
max-stable random vector $\bm{Z} \sim \mathcal{MS}(1,\nu)$.

\begin{lemma}
  Let $\bm{Z} \sim \mathcal{MS}(1, \nu)$ and $\lambda \in [0,1]$. Assume that
  $f \in \L^{p}(\pms_{1, \nu})$ for every $p \in [1, +\infty]$. Then the
  application
  \begin{align*}
    f_{\lambda} : \bm{x} \mapsto \esp\big[f\big(\lambda\bm{x} \oplus (1-\lambda)\bm{Z}\big)\big]
  \end{align*}
  is well-defined on $\Eos$, in the sense that it does not depend of the
  representative of $f$. Furthermore, it is measurable and belongs to
  $\L^{p}(\pms_{1, \nu})$.
\end{lemma}

\begin{proof}
  The mesurability and integrability properties are a consequence of Fubini's
  theorem, since
  \[
    \esp\big[\vert f \vert\big(\lambda\bm{Z} \oplus (1-\lambda)\bm{Z}'\big) \big] = \esp\big[\vert f(\bm{Z}) \vert\big] < +\infty,
  \]
  thanks to the max-stability property \eqref{prelim_stability}. Thus
  $f_{\lambda}$, which satisfies
  \[
    f_{\lambda}(\bm{x}) = \esp\big[f\big(\lambda\bm{Z} \oplus (1-\lambda)\bm{Z}'\big) \/ \bm{Z} = \bm{x}\big],
  \]
  is well-defined $\pms_{1, \nu}$-a.s. Besides, again thanks to
  \eqref{prelim_stability}, we see that if $f = g$ $\pms_{1,\nu}$-a.s., then
  $f_{\lambda} = g_{\lambda}$ $\pms_{1,\nu}$-a.s. too.
\end{proof}

We can now introduce the main object of study of this paper:

\begin{definition}[MSOU]
  The \textit{standard max-stable Ornstein-Uhlenbeck semi-group}
  $(\Pn_{t})_{t\geq 0}$ on $\L^{p}(\pms_{1 ,\nu})$ is defined by
  \begin{align}\label{4_Pt_Frcht_mltvr}
    \Pn_{t}f(\bm{x}) \coloneq \esp\big[f\big(e^{-t}\bm{x} \oplus (1-e^{-t})\bm{Z}\big)\big],\ \bm{x} \in \Eos,\ t\geq 0
  \end{align}
  where $\bm{Z} \sim \mathcal{MS}(1, \nu)$.
\end{definition}

\begin{lemma}\label{4_stationarity+}
  Let $f : \Rpd \to \R_{+}$ be a Borel $\B(\Rpd)$ measurable non-negative
  function and $p \in [1, \infty]$. Then:
  \[
    \esp\big[\Pn_{t}f(\bm{Z})\big] = \esp[f(\bm{Z})],\ f \in \L^{p}(\pms_{1 ,\nu}), t \geq 0
  \]
  where $\bm{Z} \sim \mathcal{MS}(1,\nu)$.
\end{lemma}

\begin{proof}
  Let $\bm{Z}$ and $\bm{Z}'$ be \iid random vectors, both with distribution
  $\mathcal{MS}(1,\nu)$. Then by Fubini theorem, it is clear that:
  \[
    \esp\big[\Pn_{t}f(\bm{Z})\big] = \esp\big[f(e^{-t}\bm{Z}\oplus (1-e^{-t})\bm{Z}')\big] = \esp[f(\bm{Z})],
  \]
  thanks to \eqref{prelim_stability}.
\end{proof}

%
%

\begin{proposition}

  Let $p \in [1, \infty]$. Then for every $t \in \R_{+}$ and
  $f \in \L^{p}(\pms_{1 ,\nu})$, the application $\Pn_{t}f$ belongs to
  $\L^{p}(\pms_{1 ,\nu}))$ and $\Pn_{t}$ is a contraction operator from
  $\L^{p}(\pms_{1 ,\nu})$ into itself:
  \begin{align}\label{4_contraction}
    \Vert \Pn_{t}f \Vert_{\L^{p}(\pms_{1 ,\nu})} \leq \Vert f \Vert_{\L^{p}(\pms_{1 ,\nu})}.
  \end{align}
  Furthermore, the family of operators $(\Pn_{t})_{t\geq 0}$ is a Markov
  semi-group on $\L^{p}(\pms_{1 ,\nu})$, for every $p \in [1, +\infty]$.
\end{proposition}

\begin{proof}
  By Jensen's inequality and lemma \ref{4_stationarity+}:
  \begin{align*}
    \Vert \Pn_{t}f \Vert^{p}_{\L^{p}(\pms_{1 ,\nu})} & = \esp\big[\vert \Pn_{t}f(\bm{Z}) \vert^{p}\big]                                    \\
                                                     & \leq \esp\big[\vert f \vert^{p}\big(e^{-t}\bm{Z} \oplus (1-e^{-t})\bm{Z}'\big)\big] \\
                                                     & = \esp\big[\vert f(\bm{Z}) \vert^{p}\big]                                           \\
                                                     & = \Vert f \Vert^{p}_{\L^{p}(\pms_{1 ,\nu})}.
  \end{align*}
  $\Pn_{t}$ is a linear operator on $\L^{p}(\pms_{1 ,\nu})$. Moreover $\Pn_{t}f$
  is non-negative if $f$ is non-negative, and $\Pn_{t}1 = 1$, with $1$ the
  constant function equal to 1. Besides, the semi-group relation is satisfied:
  \begin{align*}
    (\Pn_{t}\circ \Pn_{s})f(\bm{x}) & = \esp\big[(\Pn_{s})f\big(e^{-t}\bm{x} \oplus (1-e^{-t})\bm{Z}\big)\big]                                    \\
                                    & = \esp\Big[f\Big(e^{-s}\big(e^{-t}\bm{x} \oplus (1-e^{-t})\bm{Z}\big) \oplus (1-e^{-s})\bm{Z}'\Big)\Big]    \\
                                    & = \esp\Big[f\Big(e^{-(t+s)}\bm{x} \oplus \big(e^{-s}(1-e^{-t})\bm{Z}\oplus (1-e^{-s})\bm{Z}'\big)\Big)\Big] \\
                                    & = \esp\big[f\big(e^{-(t+s)}\bm{x} \oplus (1-e^{-(t+s)})\bm{Z}\big)\big]
  \end{align*}
  where $\bm{Z}'$ is an independent copy of $\bm{Z}$. Using the max-stability
  property of $\mathcal{MS}(1,\nu)$, it is clear that
  $e^{-s}(1-e^{-t})\bm{Z}\oplus (1-e^{-s})\bm{Z}' \eqdis (1-e^{-(t+s)})\bm{Z}$.
\end{proof}

\begin{remark}
  Unlike the Ornstein-Uhlenbeck semi-group, the MSOU semi-group is not
  self-adjoint: let $\bm{Z}, \bm{Z}'$ be two \iid random vectors with
  distribution $\pms_{1, \nu}$, and $f,g \in \L^{2}(\pms_{1, \nu})$. Fubini's
  theorem yields
  \begin{align*}
    \big\langle \Pn_{t}f, g \rangle_{\L^{2}(\pms_{\mu})} & = \esp\big[\Pn_{t}f(\bm{Z})g(\bm{Z})\big]                                 \\
                                                         & = \esp\Big[f\big(e^{-t}\bm{Z}\oplus (1-e^{-t})\bm{Z}'\big)g(\bm{Z})\Big].
  \end{align*}
  Thus $\Pn_{t}$ is symmetric if and only if
  \[
    \esp\Big[f\big(e^{-t}\bm{Z}\oplus (1-e^{-t})\bm{Z}'\big)g(\bm{Z})\Big] = \esp\Big[f(\bm{Z})g\big(e^{-t}\bm{Z}\oplus (1-e^{-t})\bm{Z}'\big)\Big].
  \]
  This is equivalent to asking that
  \[
    \big(e^{-t}\bm{Z}\oplus (1-e^{-t})\bm{Z}', \bm{Z}\big) \eqdis \big(\bm{Z}, e^{-t}\bm{Z}\oplus (1-e^{-t})\bm{Z}'\big).
  \]
  However the c.d.f. of the left-hand side is not symmetric as soon as
  $t \in \R_{+}^{*}$:
  \begin{align*}
    F_{(e^{-t}\bm{Z}\oplus (1-e^{-t})\bm{Z}', \bm{Z})}(\bm{x}, \bm{y}) & = \prob\big(e^{-t}\bm{Z}\oplus (1-e^{-t})\bm{Z}' \leq \bm{x}, \bm{Z} \leq \bm{y}\big)     \\
                                                                       & = \prob\big(\bm{Z} \leq e^{t}\bm{x} \odot \bm{y}, \bm{Z}' \leq (1-e^{-t})^{-1}\bm{x}\big) \\
                                                                       & = e^{-\mu[\bm{0}, e^{t}\bm{x} \odot \bm{y}]^{c}}e^{-(1-e^{-t})\mu[\bm{0}, \bm{x}]^{c}},
  \end{align*}
  where $\bm{x}, \bm{y} \in \Eos$ and $\mu$ the exponent measure of $\bm{Z}$ and
  $\bm{Z}'$.
\end{remark}
In spite of this negative result, $(\Pn_{t})_{t\geq 0}$ shares several common
points with the Ornstein-Uhlenbeck semi-group, as shown in the next result,
which is an extension of lemma \ref{4_stationarity+}.

\begin{proposition}
  $(\Pn_{t})_{t\geq 0}$ is ergodic on $\L^{p}(\pms_{1 ,\nu})$ for every
  $p \in [1, \infty]$, and its stationary measure is the multivariate Fréchet
  distribution $\mathcal{MS}(1, \nu)$.
\end{proposition}

\begin{proof}
  By the definition of $\Pn_{t}f$ and a dominated convergence argument, we get
  \[
    \Pn_{t}f(\bm{x}) \underset{t\to \infty}{\longrightarrow} \esp[f(\bm{Z})],
  \]
  which means that $(\Pn_{t})_{t\geq 0}$ is ergodic. The stationarity of
  $\pms_{1, \nu}$ for $(\Pn_{t})_{t\geq 0}$ is proved by using the exact same
  arguments than in lemma \ref{4_stationarity+}.
\end{proof}

Let us define the set of test-functions we will use to compute the generator of
$(\Pn_{t})_{t\geq 0}$:

\begin{definition}[Log-Lipschitz function]
  A function is said to be \textit{log-Lipschitz} on $\Eos$ if it belongs to
  $\mathcal{C}^{1}(\Eos)$ and satisfies
  \begin{align}\label{4_cilog_dif}
    \vert f(\bm{x}) - f(\bm{y}) \vert \leq C\Vert \log \bm{x} - \log \bm{y} \Vert_{1}.
  \end{align}
  for some constant $C>0$ and
  $\Vert \bm{x} \Vert = \vert x^{1} \vert + \dots + \vert x^{d}\vert$. The set
  of log-Lipschitz functions on $\Eos$ is denoted by $\cilog$.
\end{definition}

An easy consequence of the definition is the following.

\begin{lemma}
  A function $f$ is log-Lipschitz if and only if
  $\bm{x} \mapsto f(\exp(\bm{x}))$ is Lipschitz on $\R^{d}$, where
  $\exp\bm{x} \coloneq (\exp x^{1},\dots, \exp x^{d})$. Besides, $\cilog$
  satisfies the following:
  \begin{align}\label{4_cilog_def}
    \cilog \coloneq \Big\lb f \in \mathcal{C}^{1}(\Eos),\ \exists C > 0,\ x^{j}\vert\partial_{j}f(\bm{x})\vert \leq C\ \text{for all } \bm{x} \in \Eos\ \text{and}\ j = 1,\dots,d \Big\rb.
  \end{align}
\end{lemma}



\begin{corollary}
  Let $\bm{Z} \sim \mathcal{MS}(1,\nu)$, with $\mu$ its exponent measure. Set
  $\gamma_{t} \coloneq e^{t} - 1$.
  \begin{enumerate}
    \item For any $f\in \L^{p}(\pms_{1 ,\nu})$ and $\bm{x} \in \Eos$, we have:
          \begin{align*}\label{4_chaos_Pn}
            \Pn_{t}f( & \bm{x}) \\ &= e^{-\gamma_{t}\mu[\bm{0}, \bm{x}]^{c}}f(e^{-t}\bm{x}) + \gamma_{t}e^{-\gamma_{t}\mu[\bm{0}, \bm{x}]^{c}}\int_{[\bm{0}, \bm{x}]^{c}}f\big(e^{-t}(\bm{x}\oplus \bm{y})\big)\dif \mu(\bm{y}) + R_{t}(\bm{x}),\numberthis
          \end{align*}
          where
          \[
            R_{t}(\bm{x}) \coloneq \esp\Big[f\Big(e^{-t}\big(\bm{x} \oplus \bigoplus_{i=1}^{N_{t, \bm{x}}}\bm{Y}_{i}\big)\Big)\ind_{\lb N_{t, \bm{x}} \geq 2 \rb}\Big]
          \]
          the random variable
          $N_{t, \bm{x}} \sim \mathcal{P}(\gamma_{t}\mu[\bm{0}, \bm{x}]^{c})$
          has the Poisson $\mathcal{P}(\gamma_{t}\mu[\bm{0}, \bm{x}]^{c})$, and
          the $\bm{Y}_{i}$ are $\iid$ random variables independent of
          $N_{t, \bm{x}}$ and whose distribution is given by
          \[
            \prob(\bm{Y}_{1} \in A) = \frac{1}{\mu[\bm{0}, \bm{x}]^{c}}\mu(A),\ A \in \mathcal{B}([\bm{0}, \bm{x}]^{c}).
          \]

    \item If $f \in \cilog$, then there exists $C > 0$ such that:
          \begin{align}\label{4_remainder}
            t^{-1}\Vert R_{t} \Vert_{\L^{2}(\pms_{1,\nu})} \leq Ct^{1 - \varepsilon}.
          \end{align}
          for all $\varepsilon \in (0,1)$.
  \end{enumerate}
\end{corollary}

\begin{proof}

  1. By identity \eqref{3_chaos_MS}:
  \begin{align*}
    \frac{\gamma_{t}^{n}}{n!}\int_{([\bm{0}, \bm{x}]^{c})^{n}}f\Big(e^{-t}\big(\bm{x}\oplus & \bm{y}_{1} \oplus \dots \oplus \bm{y}_{n}\big)\Big)\prod_{i=1}^{n}\dif \mu(\bm{y}_{i})                                                                                                                                                        \\
                                                                                            & = \gamma_{t}^{n}\frac{(\mu[\bm{0}, \bm{x}]^{c})^{n}}{n!}\int_{([\bm{0}, \bm{x}]^{c})^{n}}f\Big(e^{-t}\big(\bm{x}\oplus \bm{y}_{1} \oplus \dots \oplus \bm{y}_{n}\big)\Big)\prod_{i=1}^{n}\frac{\dif \mu(\bm{y}_{i})}{\mu[\bm{0}, \bm{x}]^{c}} \\
                                                                                            & = e^{\gamma_{t}\mu[\bm{0},\bm{x}]^{c}}\esp\Big[f\Big(e^{-t}\big(\bm{x} \oplus \bigoplus_{i=1}^{N_{t, \bm{x}}} \bm{Y}_{i}\big)\Big)\ind_{\lb N_{t, \bm{x}} = n \rb}\Big].
  \end{align*}

  2. Let $t \in (0, 1]$. We start by checking that
  $\Vert \prob(N_{t, \cdot} \geq 2) \Vert_{\L^{2}(\pms_{1, \nu})} = O(t^{2})$
  when $t$ goes to $0^{+}$, \textit{i.e.} that
  \begin{align}\label{4_Ot_prob}
    \esp\big[\big(\prob(N_{t, \bm{Z}} \geq 2\/ \bm{Z})\big)^{2}\big] = O(t^{4}).
  \end{align}
  For $\bm{x} \in \Eos$, this is clear:
  \[
    \prob(N_{t, \bm{x}} \geq 2) = \gamma_{t}^{2}e^{-\gamma_{t}\mu[\bm{0}, \bm{x}]^{c}}\sum_{n=0}^{\infty}\gamma_{t}^{n}\frac{(\mu[\bm{0}, \bm{x}]^{c})^{n+2}}{(n+2)!} \leq \gamma_{t}^{2}(\mu[\bm{0}, \bm{x}]^{c})^{2}.
  \]
  By lemma \eqref{prelim_esp_mu}, we know that $(\mu[\bm{0}, \bm{Z}]^{c})^{2}$
  is integrable, hence the asymptotic relation \eqref{4_Ot_prob}. As a result,
  \eqref{4_remainder} is true if $f$ is bounded. Thus, we will assume that
  \[
    \vert f(\bm{x}) \vert \leq C\Vert\log \bm{x} \Vert_{1}.
  \]
  This entails that
  \begin{align*}
    \Big\vert \esp\Big[f\Big(e^{-t}\big(\bm{x} \oplus \bigoplus_{i=1}^{N_{t, \bm{x}}}\bm{Y}_{i}\big)\Big)\ind_{\lb N_{t, \bm{x}} \geq 2 \rb}\Big] \Big\vert & \leq C\sum_{j=1}^{d}\esp\Big[\Big\vert\log\Big(e^{-t}\big(x^{j} \oplus \bigoplus_{i=1}^{N_{t, \bm{x}}}Y^{j}_{i}\big)\Big)\Big\vert\ind_{\lb N_{t, \bm{x}} \geq 2 \rb}\Big]                         \\
                                                                                                                                                            & \leq C\sum_{j=1}^{d}\esp\Big[\Big\vert\log\Big(e^{-t}\big(x^{j} \oplus \bigoplus_{i=1}^{N_{t, \bm{x}}}Y^{j}_{i}\big)\Big)\Big\vert^{p}\Big]^{\frac{1}{p}}\prob(N_{t, \bm{x}} \geq 2)^{\frac{1}{q}} \\
                                                                                                                                                            & = C\sum_{j=1}^{d}\esp\Big[\big\vert\log\big(e^{-t}(x^{j} \oplus \gamma_{t}W^{j})\big)\big\vert^{p}\Big]^{\frac{1}{p}}\prob(N_{t, \bm{x}} \geq 2)^{\frac{1}{q}},
  \end{align*}
  the penultimate line resulting from Hölder's inequality for some
  $p, q \in (1, +\infty)$ to be determined and such that $p^{-1} + q^{-1} = 1$.
  The last line is simply the de Haan-LePage decomposition of the $j$-th
  coordinate of $\bm{x} \oplus \gamma_{t}\bm{W}$, with
  $\bm{W} \sim \mathcal{MS}(1, \nu)$. The final expectation in the former
  display is finite because
  \begin{align}\label{4_remainder_<}
    \esp\Big[\big\vert\log\big(e^{-t}(x^{j} \oplus \gamma_{t}W^{j})\big)\big\vert^{p}\ind_{\lb \gamma_{t}W^{j} \leq x^{j} \rb}\Big] \leq \big\vert t - \log x^{j}\big\vert^{p},
  \end{align}
  while on the complementary set we have instead:
  \begin{align*}\label{4_remainder_>}
    \esp\Big[\big\vert\log\big(e^{-t}(x^{j} \oplus \gamma_{t}W^{j})\big)\big\vert^{p} & \ind_{\lb \gamma_{t}W^{j} > x^{j} \rb}\Big]                                                                                             \\ &= \esp\Big[\big\vert\log\big((1-e^{-t})W^{j})\big)\big\vert^{p}\ind_{\lb \gamma_{t}W^{j} > x^{j} \rb}\Big]\\
                                                                                      & \leq 2^{p-1}\Big(\esp\big[\vert\log W^{j}\vert^{p}\big] + \vert \log(1-e^{-t}) \vert^{p}\big(1 - e^{\frac{\gamma_{t}}{x^{j}}}\big)\Big) \\
                                                                                      & \leq 2^{p-1}\Big(\esp\big[\vert\log W^{j}\vert^{p}\big] + \frac{\gamma_{t}}{x^{j}}\vert \log(1-e^{-t}) \vert^{p}\Big). \numberthis
  \end{align*}
  The right-hand sides of both \eqref{4_remainder_<} and \eqref{4_remainder_>}
  are square-integrable functions of $x^{j}$ with respect to $\pms_{1,\nu}$.
  Thus, taking squares and replacing $\bm{x}$ by
  $\bm{Z} \sim \mathcal{MS}(1, \nu)$ and independent of $\bm{W}$, one finds
  after integrating with respect to $\bm{Z}$:
  \begin{align*}
    \Vert R_{t} \Vert^{2}_{\L^{2}(\pms_{1,\nu})} & = \esp\Big[\esp\Big[f\Big(e^{-t}\big(\bm{Z} \oplus \bigoplus_{i=1}^{N_{t, \bm{Z}}}\bm{Y}_{i}\big)\Big)\ind_{\lb N_{t, \bm{Z}} \geq 2 \rb} \/ \bm{Z}\Big]^{2}\Big] \leq c\gamma_{t}^{\frac{4}{q}},
  \end{align*}
  for some constant $c > 0$. When $q$ belongs to $(1, 2)$,
  $\varepsilon = (2 - q)/q$ describes $(0,1)$, giving us the desired conclusion.
\end{proof}

The \textit{generator} $\Ln$ of $(\Pn_{t})_{t\geq 0}$ is defined as
\[
  \Ln f \coloneq \underset{t \to 0^{+}}{\lim}\frac{\Pn_{t}f - f}{t},
\]
where the convergence takes place in norm $\L^{2}(\pms_{1,\nu})$. The set of
functions $f$ such that the previous limit $\Ln f$ exists is called the
\textit{domain} of $\Ln$ and will be denoted by $\dom(\Ln)$. For more about
those notions, we refer to \cite{Ethier86}. Our next results proves that
$\cilog$ is included in $\dom(\Ln)$.

\begin{proposition} Let $\bm{Z} \sim \mathcal{MS}(1, \nu)$, with exponent
  measure $\mu$. The Markov semi-group $(\Pn_{t})_{t\geq 0}$ has generator
  $\Ln$, given for any $\ f\in \cilog$ by:
  \begin{align}\label{4_L}
    \Ln f(\bm{x}) = -\langle \bm{x},\nabla f(\bm{x})\rangle + \int_{\Eo}\big(f(\bm{x}\oplus \bm{y}) - f(\bm{x})\big)\dif\mu(\bm{y}),\ \bm{x} \in \Eos
  \end{align}
  and where $\langle \cdot,\cdot \rangle$ denotes the standard Euclidean inner
  product on $\R^{d}$.
\end{proposition}

\begin{proof} Let $f \in \cilog$. As in the proof of \eqref{4_remainder}, the
  result is easier to prove if $f$ is bounded, so we will also assume that
  $\vert f(\bm{x}) \vert \leq C\Vert \log\bm{x}\Vert_{1}$. Denote provisionally
  by $\mathscr{L}$ the right-hand side of \eqref{4_L}. We must prove that
  \[
    \Big\Vert \frac{\Pn_{t}f - f}{t} - \mathscr{L} f \Big\Vert^{2}_{\L^{2}(\pms_{1,\nu})} \underset{t \to 0^{+}}{\longrightarrow} 0,
  \]
  where $\bm{Z} \sim \mathcal{MS}(1, \nu)$. We have:
  \begin{align}\label{4_gene_dcpt}
    \Big\Vert \frac{\Pn_{t}f - f}{t} - \mathscr{L} f \Big\Vert^{2}_{\L^{2}(\pms_{1,\nu})} \leq C\big(A + B + t^{-2}\Vert R_{t} \Vert^{2}_{\L^{2}(\pms_{1,\nu})}\big)
  \end{align}
  for some constant $C > 0$, with $R_{t}$ the remainder term in identity
  \eqref{4_remainder},
  \begin{align*}
    A \coloneq \frac{1}{t^{2}}\esp\Big[\Big(e^{-\gamma_{t}\mu[\bm{0}, \bm{Z}]^{c}}f(e^{-t}\bm{Z}) - f(\bm{Z}) + t\langle \bm{Z}, \nabla f(\bm{Z}) \rangle \Big)^{2}\Big]
  \end{align*}
  and
  \begin{align*}
    B \coloneq \frac{1}{t^{2}}\esp\Big[\Big( \gamma_{t}e^{-\gamma_{t}\mu[\bm{0}, \bm{Z}]^{c}}\int_{[\bm{0}, \bm{Z}]^{c}}f\big(e^{-t}(\bm{Z}\oplus \bm{y})\big)\dif\mu(\bm{y}) - \int_{[\bm{0}, \bm{Z}]^{c}}f(\bm{Z} \oplus \bm{y})\dif\mu(\bm{y}) \Big)^{2}\Big]
  \end{align*}

  1. Observe that
  \begin{multline*}
    e^{-\gamma_{t}\mu[\bm{0}, \bm{Z}]^{c}}f(e^{-t}\bm{Z}) - f(\bm{Z}) + t\big\langle \bm{Z}, \nabla f(\bm{Z})\big\rangle + t\mu[\bm{0}, \bm{Z}]^{c}f(\bm{Z})\\
    = \big(e^{-\gamma_{t}\mu[\bm{0}, \bm{Z}]^{c}} - 1 + \gamma_{t}\mu[\bm{0}, \bm{Z}]^{c}\big)f(e^{-t}\bm{Z}) + (\gamma_{t} - t)\mu[\bm{0}, \bm{Z}]^{c}\\
    + f(e^{-t}\bm{Z}) - f(\bm{Z}) + t\big\langle \bm{Z}, \nabla f(\bm{Z})\big\rangle
  \end{multline*}
  Clearly $\gamma_{t} - t$ is of order $t^{2}$, while the inequality
  \[
    \vert e^{-x} - 1 + x \vert \leq \frac{x^{2}}{2},\ x \geq 0
  \]
  implies that the term between parentheses is bounded by:
  \begin{align*}
    \frac{1}{t^{2}}\esp\Big[\Big(\big(e^{-\gamma_{t}\mu[\bm{0}, \bm{Z}]^{c}} - 1 + \gamma_{t}\mu[\bm{0}, \bm{Z}]^{c}\big)f(e^{-t}\bm{Z}) & \Big)^{2}\Big]                                                                                                          \\
                                                                                                                                         & \leq \frac{1}{2t^{2}}C^{2}\gamma_{t}^{4}\esp\Big[(\mu[\bm{0}, \bm{Z}]^{c})^{4}\Vert t - \log \bm{Z} \Vert_{1}^{2}\Big].
  \end{align*}
  A Cauchy-Schwartz argument coupled with lemma \ref{prelim_esp_mu} show that
  this last expectation is finite. Taylor's formula applied to the class
  $\mathscr{C}^{1}$ function $f$ between $e^{-t}\bm{Z}$ and $\bm{Z}$ gives us
  the existence of a function $h_{t} : \R^{d} \to \R^{d}$ such that:
  \begin{align*}
    f(e^{-t}\bm{Z}) = f(\bm{Z}) + (e^{-t} - 1)\big\langle \bm{Z} + (e^{-t} - 1)\big\langle \bm{Z}, h_{t}(e^{-t}\bm{Z})\big\rangle
  \end{align*}
  with $h_{t}(e^{-t}\bm{Z})$ vanishing as $t$ goes to $0$. Because $f$ is
  log-Lipschitz, the following inequality holds
  \[
    (1 - e^{-t})^{2}\esp\big[\big\langle \bm{Z}, h_{t}(e^{-t}\bm{Z})\big\rangle^{2}\big] \leq 2\big(t^{2} + (1-e^{-t})^{2}\big).
  \]
  Therefore, a dominated convergence argument yields
  \[
    \frac{1}{t^{2}}\esp\big[\big(f(e^{-t}\bm{Z}) - f(\bm{Z}) + (1-e^{-t})\big\langle \bm{Z}, \nabla f(\bm{Z})\big\rangle\big)^{2}\big] = \Big(\frac{1 - e^{-t}}{t}\Big)^{2}\esp\big[\big(h_{t}(e^{-t}\bm{Z})\big)^{2}\big] \underset{t \to 0^{+}}{\longrightarrow} 0.
  \]
  Replacing $1-e^{-t}$ by $t$ before the inner product yields the same bound, so
  that all in all:
  \[
    A \leq C\frac{1}{t^{2}}\gamma_{t}^{4} + C\big(t^{2} + (1-e^{-t})^{2}\big).
  \]

  2. We have the decomposition
  \begin{align*}
    \gamma_{t}e^{-\gamma_{t}\mu[\bm{0}, \bm{Z}]^{c}}\int_{[\bm{0}, \bm{Z}]^{c}}f\big(e^{-t}(\bm{Z}\oplus \bm{y})\big) & \dif\mu(\bm{y}) - t\int_{[\bm{0}, \bm{Z}]^{c}}f(\bm{Z} \oplus \bm{y})\dif\mu(\bm{y})                                                             \\
                                                                                                                      & = \big(\gamma_{t}e^{-\gamma_{t}\mu[\bm{0}, \bm{Z}]^{c}} - t\big)\int_{[\bm{0}, \bm{Z}]^{c}}f\big(e^{-t}(\bm{Z}\oplus \bm{y})\big)\dif\mu(\bm{y}) \\
                                                                                                                      & \quad + t\Big(\int_{[\bm{0}, \bm{Z}]^{c}}f\big(e^{-t}(\bm{Z}\oplus \bm{y})\big) - f(\bm{Z} \oplus \bm{y})\dif\mu(\bm{y})\Big)
  \end{align*}
  The second part is bounded by:
  \[
    t\Big\vert\int_{[\bm{0}, \bm{Z}]^{c}}f\big(e^{-t}(\bm{Z}\oplus \bm{y})\big) - f(\bm{Z} \oplus \bm{y})\dif\mu(\bm{y})\Big\vert \leq t^{2}\mu[\bm{0}, \bm{Z}]^{c},
  \]
  while the first is of order $t^{2}$ as well, since
  \[
    \gamma_{t}e^{-\gamma_{t}\mu[\bm{0}, \bm{Z}]^{c}} - t \underset{t \to 0^{+}}{\sim} -t^{2}\mu[\bm{0}, \bm{Z}]^{c}.
  \]
  Besides,
  $\mu[\bm{0},\bm{Z} ]^{c}\int_{[\bm{0}, \bm{Z}]^{c}}f(e^{-t}(\bm{Z}\oplus \bm{y}))\dif\mu(\bm{y})$
  is square-integrable thanks to the fact that $f$ is log-Lipschitz and
  Cauchy-Schwarz inequality, hence:
  \[
    B \leq Ct^{2}.
  \]
\end{proof}

The right-inverse of $\Ln$ is well-defined if $f \in \cilog$ and
$\esp[f(\bm{Z})] = 0$. To prove this, we first need a lemma.

\begin{lemma}\label{4_cilog_invar}
  Let $f \in \cilog$. Then $\Pn_{t}f \in \cilog$.
\end{lemma}

\begin{proof}
  The function $\bm{x} \mapsto f(e^{-t}\bm{x} \oplus (1-e^{-t})\bm{Z})$ is
  $\pms_{1,\nu}$-a.s. differentiable because $\pms_{1,\nu}$ is diffuse. For any
  $j \in \lbra 1,d \rbra$ and $x^{j} \in [a, b] \subseteq \Rps$, one has
  \begin{align*}
    \big\vert\partial_{j}f\big(e^{-t}\bm{x} \oplus (1- & e^{-t})\bm{Z}\big)\big\vert                                                                                                                                                             \\
                                                       & = e^{-t}\big\vert(\partial_{j}f)\big(e^{-t}\bm{x} \oplus (1-e^{-t})\bm{Z}\big)\big\vert\ind_{\lb x^{j} \geq \gamma_{t}Z^{j}\rb}                                                         \\
                                                       & \leq \frac{1}{x^{j}}\big(e^{-t}x^{j} \oplus (1-e^{-t})Z^{j}\big)\big\vert(\partial_{j}f)\big(e^{-t}\bm{x} \oplus (1-e^{-t})\bm{Z}\big)\big\vert\ind_{\lb x^{j} \geq \gamma_{t}Z^{j}\rb} \\
                                                       & \leq \frac{1}{a}\cdotp
  \end{align*}
  By a dominated convergence argument, we deduce that $\Pn_{t}f$ is
  differentiable. Thanks to the previous display, one sees that $\Pn_{t}f$ is
  log-Lipschitz, with partial derivatives equal to
  \begin{align}\label{4_partial_j}
    \partial_{j}\Pn_{t}f(\bm{x}) = e^{-t}\esp\big[(\partial_{j}f)\big(e^{-t}\bm{x} \oplus (1-e^{-t})\bm{Z}\big)\ind_{\lb x^{j} \geq \gamma_{t}Z^{j}\rb}\big].
  \end{align}
  The continuity of $\partial_{j}\Pn_{t}f$ is once again a consequence of the
  fact that $\pms_{1,\nu}$ is diffuse:
  \[
    \ind_{\lb x^{j} \geq \gamma_{t}Z^{j}\rb} = 1 - \ind_{\lb x^{j} \leq \gamma_{t}Z^{j}\rb}\ \pms_{1,\nu}\text{-a.s.}
  \]
  the left-hand side being right-continuous, while the right-hand side is
  left-continuous.
\end{proof}

\begin{proposition}\label{4_Ln-1}
  Let $f \in \cilog$. Then there exists a function denoted by
  $\Ln^{-1}f \in \dom(\Ln)$ such that
  \[
    \Ln(\Ln^{-1}f) = f - \esp[f(Z)]\ \pms_{1,\nu}\text{-a.s.}
  \]
  with $\bm{Z} \sim \mathcal{MS}(1, \nu)$. It is given by:
  \begin{align}\label{4_Ln-1_eq}
    \Ln^{-1}f = -\int_{0}^{\infty}\big(\Pn_{t}f - \esp[f(\bm{Z})]\big)\dif t,
  \end{align}
  where the integral is defined as the limit in norm $\L^{2}(\pms_{1,\nu})$ of
  $\int_{0}^{n}(\Pn_{t}f - \esp[f(\bm{Z})]) \dif t$ when $n$ goes to infinity.
  Besides, $\Ln^{-1}f$ is differentiable on $\Eos$ and the following inequality
  holds for some $C > 0$ independent of $\bm{x}$:
  \begin{align}\label{4_diff_Ln-1}
    x^{j}\vert\partial_{j}\Ln^{-1}f(\bm{x})\vert \leq C\int_{0}^{\infty} e^{-\frac{\gamma_{t}}{x^{j}}} \dif t = \int_{0}^{\infty}\frac{x^{j}}{x^{j}t + 1}e^{-t} \dif t,\ j \in \lbra 1,d \rbra.
  \end{align}
\end{proposition}

\begin{proof}
  Let $\bm{W}, \bm{Z}$ be \iid random vectors with common distribution
  $\mathcal{MS}(1, \nu)$. First, for $f \in \cilog$, we have thanks to Jensen's
  inequality:
  \begin{align*}
    \Vert \Pn_{t}f - \esp[f(\bm{Z})] \Vert^{2}_{\L^{2}(\pms_{1,\nu})} & \leq \esp\Big[\Big(\sum_{j=1}^{d}\big\vert \log\big(e^{-t}W^{j} \oplus (1-e^{-t})Z^{j}\big) - \log Z^{j} \big\vert\Big)^{2}\Big] \\
                                                                      & \leq dC\sum_{j=1}^{d}\esp\Big[\big(\log\big(e^{-t}W^{j} \oplus (1-e^{-t})Z^{j}\big) - \log Z^{j} \big)^{2}\Big]                  \\
                                                                      & = d^{2}C\esp\Big[\big(\log\big(e^{-t}W \oplus (1-e^{-t})Z\big) - \log Z \big)^{2}\Big],
  \end{align*}
  where $W, Z$ are two \iid random variables with distribution $\mathcal{F}(1)$.
  Depending on whether $\gamma_{t}Z \leq W$ or $\gamma_{t}Z > W$, the previous
  expectation simplifies as
  \begin{align*}\label{4_norm_P_t}
    \esp\Big[\big(\log\big(e^{-t}W & \oplus (1-e^{-t})Z\big) - \log Z \big)^{2}\Big]                                                                                             \\
                                   & = \esp\Big[\big(\log e^{-t}W - \log Z \big)^{2}\ind_{\lb\gamma_{t}Z \leq W\rb}\Big] + \big(\log(1-e^{-t})\big)^{2}\prob(\gamma_{t}Z > W)    \\
                                   & \leq 3\esp\Big[\big(t^{2} + (\log W)^{2} + (\log Z)^{2}\big)\ind_{\lb\gamma_{t}Z \leq W\rb}\Big] + (1 - e^{-t})\big(\log(1-e^{-t})\big)^{2} \\
                                   & \leq 3\big(t^{2}e^{-t} + Ce^{-\frac{t}{2}}\big) + (1 - e^{-t})\big(\log(1-e^{-t})\big)^{2}, \numberthis
  \end{align*}
  for some constant $C > 0$. We have used Cauchy-Schwarz inequality as well as
  the fact that $\exp(-1/Z)$ has the standard uniform distribution, so that
  \[
    \prob(\gamma_{t}Z > W) = \esp\big[e^{-\frac{1}{\gamma_{t}Z}}\big] = \frac{\gamma_{t}}{\gamma_{t} + 1} = 1 - e^{-t}.
  \]
  The right-hand side of \eqref{4_norm_P_t} is an integrable function of $t$ on
  $\R_{+}$. Furthermore, $(\int_{0}^{n}\Pn_{t}f^{*}\dif t)_{n\geq 0}$ is a
  Cauchy sequence of $\L^{2}(\pms_{1,\nu})$, with $f^{*} = f - \esp[f(\bm{Z})]$:
  first, $\int_{0}^{n}\Pn_{t}f^{*}\dif t$ has a sense in $\L^{2}(\pms_{1,\nu})$,
  since $t \mapsto \Pn_{t}f^{*}$ is continuous for that topology. Next we have
  \[
    \Big\Vert\int_{0}^{n+m}\Pn_{t}f^{*}\dif t - \int_{0}^{n}\Pn_{t}f^{*}\dif t \Big\Vert_{\L^{2}(\pms_{1,\nu})} \leq \int_{n}^{n+m}\Vert \Pn_{t}f^{*} \Vert_{\L^{2}(\pms_{1,\nu})} \underset{n,m \to \infty}{\longrightarrow} 0,
  \]
  as the remainder of a converging integral. Let us denote by
  \[
    \int_{0}^{\infty} \Pn_{t}f^{*} \dif t \coloneq \underset{n \to \infty}{\lim}\int_{0}^{n}\Pn_{t}f^{*} \dif t
  \]
  the limit of that sequence. By continuity of $\Pn_{t}$, one has:
  \[
    \Pn_{s}\Big(\int_{0}^{\infty} \Pn_{t}f^{*} \dif t\Big) = \underset{n \to \infty}{\lim}\Pn_{s}\Big(\int_{0}^{n}\Pn_{t}f^{*} \dif t\Big) = \int_{0}^{\infty}\Pn_{t+s}f^{*} \dif t = \int_{s}^{\infty}\Pn_{t}f^{*} \dif t.
  \]
  From the last display, one deduces
  \begin{align*}
    \frac{\Pn_{s}\Big(\int_{0}^{\infty} \Pn_{t}f^{*} \dif t \Big) - \int_{0}^{\infty} \Pn_{t}f^{*} \dif t}{s} = \frac{1}{s}\int_{0}^{s}\Pn_{t}f \dif t \underset{t \to 0^{+}}{\longrightarrow} f.
  \end{align*}
  This means exactly that $\Ln(\int_{0}^{\infty} \Pn_{t}f^{*} \dif t) = f$.

  Inequality \eqref{4_diff_Ln-1} is a straightforward consequence of
  \eqref{4_Ln-1_eq} and \eqref{4_partial_j}:
  \begin{align*}
    x^{j}\vert\partial_{j}\Ln^{-1}f(\bm{x})\vert \leq C\int_{0}^{\infty} \vert\partial_{j}\Pn_{t}f(\bm{x})\vert\dif t \leq C\int_{0}^{\infty} \prob(\gamma_{t}Z \leq x^{j})\dif t,
  \end{align*}
  since $\Pn_{t}f$ belongs to $\cilog$.
\end{proof}

The integral operator in the generator $\Ln$ satisfies several properties. We
need two lemmas first.

\begin{lemma}
  Define the transformation $T$ by
  \begin{align*}
    \begin{array}{ccccc}
      T & : & \Eos \times \Eo  & \longrightarrow & \Eos                 \\
        &   & (\bm{x}, \bm{y}) & \longmapsto     & \bm{x}\oplus \bm{y}.
    \end{array}
  \end{align*}
  Then $T_{*}(\pms_{1,\nu} \otimes \mu)$ is absolutely continuous with respect
  to $\pms_{1, \nu}$.
\end{lemma}

\begin{proof}
  It is well-known that the discrete gradient is \textit{closable}: if
  $\bar{f}, \bar{g} : \mathfrak{N}_{\Eo} \to \R$ are two functions on the space
  of configurations of $\Eo$, then the implication
  \[
    \bar{f}(\phi) = \bar{g}(\phi) \dif\prob_{\eta}(\phi)-\text{a.s.} \Longrightarrow D_{\bm{y}}\bar{f}(\phi) = D_{\bm{y}}\bar{g}(\phi) \dif(\prob_{\eta}\otimes \mu)(\phi, \bm{y})-\text{a.s.}
  \]
  holds, where $\prob_{\eta}$ denotes the distribution of a Poisson process
  $\eta$ on $\Eo$ with intensity measure $\mu$. This is a consequence of the
  Campbell-Mecke formula, see \cite{Decreusefond22} for instance. One deduces
  that if $\bar{f}(\phi) = \bar{g}(\phi)$ $\dif\prob_{\eta}(\phi)$-a.s., then
  $\bar{f}(\phi + \delta_{\bm{y}}) = \bar{g}(\phi + \delta_{\bm{y}}) \dif(\prob_{\eta}\otimes \mu)(\phi, \bm{y})-\text{a.s.}$.
  Taking $\bar{f} = \ind_{A}$, with $\prob_{\eta}(A) = 0$, and $\bar{g} = 0$,
  one infers that
  \[
    \int_{\Eo}\esp\big[\ind_{A}(\eta + \delta_{\bm{y}})\big] \dif\mu(\bm{y}) = 0,
  \]
  \textit{i.e.} that $T'_{*}(\prob_{\eta} \otimes \mu)$ is absolutely continuous
  with respect to $\prob_{\eta}$, with
  $T'(\phi, \bm{y}) \coloneq \phi + \delta_{\bm{y}}$. Specializing this result
  to functionals of the form $f = \bar{f} \circ \m$ and events $A$ depending
  only on $\m(\phi)$, we get the announced statement.
\end{proof}

\begin{lemma}
  Let $f \in \L^{\infty}(\pms_{1,\nu})$. Then
  $(\bm{x}, \bm{y}) \mapsto f(\bm{x} \oplus \bm{y})$ is bounded
  $(\pms_{1,\nu} \otimes \mu)$-a.s. by
  $\Vert f \Vert_{\L^{\infty}(\pms_{1,\nu})}$. Consequently, the integral
  \[
    \int_{\Eo}\big\vert f(\bm{x}\oplus \bm{y}) - f(\bm{x})\big\vert\dif\mu(\bm{y})
  \]
  is finite $\pms_{1,\nu}$-a.s and belongs to $\L^{p}(\pms_{1,\nu})$ for every
  $p \in [1, +\infty)$.
\end{lemma}

\begin{proof}
  The previous lemma implies that if $f$ is bounded $\pms_{1,\nu}$-a.s. by
  $\Vert f \Vert_{\L^{\infty}(\pms_{1, \nu})}$, then so is
  $(\bm{x}, \bm{y}) \mapsto f(\bm{x} \oplus \bm{y})$
  $(\pms_{1,\nu} \otimes \mu)$-a.s., hence the first part of the result. By
  Jensen's inequality
  \begin{align*}\label{4_Dn_inequality}
    \esp\Big[\Big(\int_{\Eo} \big\vert f(\bm{Z} \oplus \bm{y}) - f(\bm{Z})\big\vert & \dif\mu(\bm{y})\Big)^{p}\Big]                                                                                                                                         \\ &\leq \int_{\Eo} \esp\big[(\mu[\bm{0}, \bm{Z}]^{c})^{p-1}\big\vert f(\bm{Z} \oplus \bm{y}) - f(\bm{Z})\big\vert^{p}\big]\dif\mu(\bm{y})\\
                                                                                    & = \int_{\Eo} \esp\big[(\mu[\bm{0}, \bm{Z}]^{c})^{p-1}\big\vert f(\bm{Z} \oplus \bm{y}) - f(\bm{Z})\big\vert^{p} \ind_{\lb \bm{y} \nleq \bm{Z}\rb}\big]\dif\mu(\bm{y}) \\
                                                                                    & \leq 2^{p}\Vert f \Vert_{\L^{\infty}(\pms_{1,\nu})}^{p}\esp\big[(\mu[\bm{0},\bm{Z}]^{c})^{p} \big], \numberthis
  \end{align*}
  and we know that the last expectation is finite thanks to lemma
  \ref{prelim_esp_mu}.
\end{proof}

\begin{proposition}
  For $f \in \L^{\infty}(\pms_{1,\nu})$ and $\bm{x}\in \Eos$, set:
  \begin{align*}
    \Dn f(\bm{x}) \coloneq \int_{\Eo}\big(f(\bm{x}\oplus \bm{y}) - f(\bm{x})\big)\dif\mu(\bm{y}) = \int_{\Epol}\big(f(\bm{x}\oplus r\bm{u}) - f(\bm{x})\big)\frac{1}{r^{2}}\dif r \dif\nu(\bm{u}).
  \end{align*}
  The operator $\Dn$ is continuous from $\L^{\infty}(\pms_{1,\nu})$ to
  $\L^{p}(\pms_{1, \nu})$ for every $p \in [1, \infty)$:
  \[
    \Vert \Dn f\Vert_{\L^{p}(\pms_{1,\nu})} \leq 2\Vert \mu[\bm{0}, \cdot]^{c}\Vert_{\L^{p}(\pms_{1,\nu})}\Vert f\Vert_{\L^{\infty}(\pms_{1,\nu})}.
  \]
  It admits the following alternative expressions on $\cilog$.
  \begin{align}
    \Dn f(\bm{x}) & = \int_{\Eo} \big\langle \bm{y},\nabla f(\bm{x}\oplus \bm{y})\big\rangle_{\bm{x}} \dif\mu(\bm{y}),\label{4_Dn_alt}                         \\
                  & = \sum_{j=1}^{d}\int_{\lb ru^{j} > x^{j}\rb} u^{j}\partial_{j}f(\bm{x}\oplus r\bm{u})\frac{1}{r}\dif r\dif \nu(\bm{u})\label{4_Dn_alt_pol}
  \end{align}
  where for all $\bm{x} \in \Eos$ and $\bm{y}, \bm{z} \in \Eo$
  \[
    \langle \bm{y}, \bm{z} \rangle_{\bm{x}} \coloneq \sum_{j=1}^{d}y^{j}z^{j}\ind_{\lb y^{j} > x^{j} \rb}
  \]
  and $\lb ru^{j} > x^{j}\rb$ is the subset of $\Epol$ of $(r, \bm{u})$ such
  that $ru^{j} > x^{j}$.
\end{proposition}

\begin{proof}
  The continuity of $\Dn$ is a straightforward consequence of
  \eqref{4_Dn_inequality}. This operator can be expressed as an integral on
  either $\Eo$ and $\Epol$ thanks to the polar decomposition. Using the latter
  we find:
  \begin{align*}
    \Dn f(\bm{x}) & = \int_{\lb r\bm{u} \nleq \bm{x} \rb}\big(f(\bm{x}\oplus r\bm{u}) - f(\bm{x})\big)\frac{1}{r^{2}}\dif r\dif \nu(\bm{u})                           \\
                  & = \int_{\S_{+}^{d-1}}\int_{\min\frac{\bm{x}}{\bm{u}}}^{\infty}\big(f(\bm{x}\oplus r\bm{u}) - f(\bm{x})\big)\frac{1}{r^{2}}\dif r\dif \nu(\bm{u}).
  \end{align*}
  because $r\bm{u}$ is not less than $\bm{x}$ only if $r$ is greater than the
  smallest coordinate of $\bm{x}/\bm{u}$. An integration-by-parts yields:
  \begin{align*}
    \int_{\S_{+}^{d-1}}\int_{\min\frac{\bm{x}}{\bm{u}}}^{\infty}\big(f(\bm{x}\oplus r\bm{u}) - f(\bm{x})\big)\frac{1}{r^{2}}\dif r\dif \nu(\bm{u}) & = \int_{\S_{+}^{d-1}}\int_{\min\frac{\bm{x}}{\bm{u}}}^{\infty} \big\langle \bm{u},\nabla f(\bm{x} \oplus r\bm{u})\big\rangle_{\bm{x}} \frac{1}{r}\dif r\dif \nu(\bm{u}) \\
                                                                                                                                                   & = \int_{\lb r\bm{u} \nleq \bm{x} \rb} \big\langle r\bm{u},\nabla f(\bm{x}\oplus r\bm{u})\big\rangle_{\bm{x}} \frac{1}{r^{2}}\dif r\dif \nu(\bm{u})                      \\
                                                                                                                                                   & = \int_{\Epol} \big\langle r\bm{u},\nabla f(\bm{x}\oplus r\bm{u})\big\rangle_{\bm{x}} \frac{1}{r^{2}}\dif r\dif \nu(\bm{u}) \numberthis                                 \\
                                                                                                                                                   & = \int_{\Eo} \big\langle \bm{y},\nabla f(\bm{x}\oplus \bm{y})\big\rangle_{\bm{x}} \dif\mu(\bm{y}),
  \end{align*}
  hence equality \eqref{4_Dn_alt}.

  The second alternative expression \eqref{4_Dn_alt_pol} of $\Dn$ is an
  immediate consequence of \eqref{4_Dn_alt}, as well of the definition of
  $\langle \cdot,\cdot \rangle_{\bm{x}}$ and of the Euclidean inner product.
\end{proof}

\begin{example}

  Set $h_{\bm{z}} \coloneq \ind_{(-\bm{\infty}, \bm{z}]}$, for
  $\bm{z} \in \Eos$. This function belongs to $\L^{\infty}(\pms_{1,\nu})$ for
  every choice of angular measure $\nu$ and is a \textit{character} for the max
  operation $\oplus$ (see \cite{Davydov08}):
  \[
    h_{\bm{z}}(\bm{x}\oplus \bm{y}) = h_{\bm{z}}(\bm{x})h_{\bm{z}}(\bm{y}).
  \]
  Let $\bm{Z} \sim \mathcal{MS}(1, \nu)$. The function $h_{\bm{z}}$ satisfies
  \[
    \Dn h_{\bm{z}}(\bm{x}) = -h_{\bm{z}}(\bm{x})\int_{\Eo}\big(1 - \ind_{(-\bm{\infty}, \bm{z}]}(\bm{y})\big) \dif\nu(\bm{y}) = -\mu[\bm{0}, \bm{z}]^{c}h_{\bm{z}}(\bm{x}),
  \]
  so that $h_{\bm{z}}$ is an eigenfunction of $\Dn$, with associated eigenvalue
  $\lambda_{z} = -\mu[\bm{0}, \bm{z}]^{c} \leq 0$.
\end{example}

Depending on the angular measure, the regularity of $\Dn f$ changes drastically,
even for smooth $f$, as the next examples show.

\begin{example} For the sake of clarity, assume that the reference norm is the
  infinity norm $\Vert\cdot\Vert_{\infty}$ on $\R^{d}$ and $f \in \cic$.
  \begin{itemize}
    \item[-] In the case of complete independence,
          $\nu = \sum_{j=1}^{d}\delta_{\bm{e}_{j}}$, where $\bm{e}_{j}$ is the
          $j$-th vector of the canonical basis of $\R^{d}$, so that:
          \begin{align*}
            \Dn f(\bm{x}) & = \sum_{j=1}^{d}\int_{0}^{\infty}\big(f(\bm{x} \oplus r\bm{e}_{j}) - f(\bm{x})\big)\frac{1}{r^{2}}\dif r           \\
                          & = \sum_{j=1}^{d}\int_{x^{j}}^{\infty} \partial_{j}f(\bm{x} \oplus r\bm{e}_{j})\frac{1}{r}\dif r,\ \bm{x} \in \Eos.
          \end{align*}
          Notice that $\Dn f$ is still infinitely differentiable with respect to
          each $x^{j}$. This stems from the specific shape of the angular
          measure.

    \item[-] On the other hand, in the case of complete dependence,
          \textit{i.e.} $\nu = \delta_{\bm{1}}$, we get:
          \begin{align*}
            \Dn f(\bm{x}) & = \int_{\min \bm{x}}^{\infty}\big(f(\bm{x} \oplus r\bm{1}) - f(\bm{x})\big)\frac{1}{r^{2}} \dif r              \\
                          & = \sum_{j=1}^{d}\int_{x^{j}}^{\infty} \partial_{j}f(\bm{x} \oplus r\bm{1})\frac{1}{r}\dif r,\ \bm{x} \in \Eos,
          \end{align*}
          where $\bm{1} = (1,\dots,1)$. $\Dn f$ remains differentiable with
          respect to each $x^{j}$ once but not more in general.
  \end{itemize}
\end{example}

It is well-known that the Gaussian Ornstein-Uhlenbeck semi-group
$(P_{t})_{t\geq 0}$ satisfies the following commutation rule:
\begin{align}\label{4_Bakry}
  \nabla P_{t}f(\bm{x}) = e^{-t}P_{t}\nabla f(\bm{x}),\ \bm{x} \in \R^{d},\ t \geq 0
\end{align}
where $f$ belongs to (say) the Schwartz class $\mathcal{S}(\R^{d})$ and $\nabla$
denotes the gradient operator. A similar relation holds true for the MSOU
semi-group, although the gradient is replaced with the operator $\Dn$.

\begin{proposition}
  The operator $\Dn$ satisfies the following.
  \begin{enumerate}
    \item (Commutation rule) For all $f \in \L^{\infty}(\pms_{1},\nu)$, we have:
          \begin{align}\label{4_comm_Frcht_d}
            \Dn \Pn_{t}f = e^{-t}\Pn_{t}\Dn f,\ t \geq 0.
          \end{align}
    \item (Infinitesimal commutation rule) For all
          $f \in \mathcal{C}^{1}_{c}(\Eo)$, the following identity holds true
          \begin{align}\label{4_inf_comm_Frcht_d} [\Ln, \Dn]f = \Dn f,
          \end{align}
          where $[A, B] = A \circ B - B \circ A$ if $A$ and $B$ are two
          operators.

  \end{enumerate}
\end{proposition}

\begin{proof}

  1. If $f \in \L^{\infty}(\pms_{1,\nu})$, then $\Dn f$ belongs to
  $\L^{p}(\pms_{1, \nu})$ for every $p \in [1, \infty)$, so the composition
  $\Pn_{t}\Dn f$ is well-defined. We find:
  \begin{align*}
    \Pn_{t} & \Dn f(\bm{x})                                                                                                                                                                    \\ &= \esp\Big[(\Dn f)\big(e^{-t}\bm{x}\oplus (1-e^{-t})\bm{Z}\big)\Big]\\
            & = \esp\Big[\int_{\Epol}\Big(f\big(e^{-t}\bm{x}\oplus (1-e^{-t})\bm{Z}\oplus r\bm{u}\big) - f(e^{-t}\bm{x}\oplus (1-e^{-t})\bm{Z})\Big)\frac{1}{r^{2}}\dif r\dif \nu(\bm{u})\Big] \\
            & = \int_{\Epol}\esp\Big[f\big(e^{-t}\bm{x}\oplus r\bm{u}\oplus (1-e^{-t})\bm{Z}\big) - f(e^{-t}\bm{x}\oplus (1-e^{-t})\bm{Z})\Big]\frac{1}{r^{2}}\dif r\dif \nu(\bm{u}).
  \end{align*}
  On the other hand, a change of variable yields:
  \begin{align*}
    \Dn & \Pn_{t}f(\bm{x})                                                                                                                                                                       \\ &= \int_{\Epol}\Big(\Pn_{t}f(\bm{x}\oplus r\bm{u}) - \Pn_{t}f(\bm{x})\Big)\frac{1}{r^{2}}\dif r\dif \nu(\bm{u})\\
        & = \int_{\Epol}\esp\Big[\Big(f\big(e^{-t}(\bm{x}\oplus r\bm{u})\oplus (1-e^{-t})\bm{Z}\big) - f(e^{-t}\bm{x}\oplus (1-e^{-t})\bm{Z})\Big)\frac{1}{r^{2}}\dif r\dif \nu(\bm{u})\Big]     \\
        & = e^{-t}\int_{\Epol}\esp\Big[\Big(f\big(e^{-t}\bm{x}\oplus r\bm{u}\oplus (1-e^{-t})\bm{Z}\big) - f(e^{-t}\bm{x}\oplus (1-e^{-t})\bm{Z})\Big)\frac{1}{r^{2}}\dif r\dif \nu(\bm{u})\Big] \\
        & = e^{-t}\Pn_{t}\Dn f(\bm{x}).
  \end{align*}

  2. Equation \eqref{4_Dn_alt_pol} makes it clear that $\Ln\Dn$ and $\Dn\Ln$ are
  well-defined for every choice of angular measure. The same goes for $\Pn\Dn$
  and $\Dn\Pn$ thanks to the commutation rule. We write
  \begin{align*}
    [\Ln, \Dn] & = \underset{t \to 0^{+}}{\lim}\frac{1}{t}[\Pn_{t} - \mathrm{Id}, \Dn]      \\
               & = \underset{t \to 0^{+}}{\lim}\frac{1}{t}[\Pn_{t}, \Dn]                    \\
               & = \underset{t \to 0^{+}}{\lim}\frac{1}{t}\big(\Pn_{t}\Dn - \Dn\Pn_{t}\big) \\
               & = \underset{t \to 0^{+}}{\lim}\frac{1 - e^{-t}}{t}\Pn_{t}\Dn = \Dn.
  \end{align*}
\end{proof}

The second commutation rule between $\Ln$ and $\Dn$ corresponds to an
"infinitesimal version" of \eqref{4_comm_Frcht_d}, to quote the expression of
\cite{Chafaï06}, page 6.

The operator $\Dn$ is part of a functional characterization of simple max-stable
distributions.

\begin{proposition}

  Let $\bm{Z}$ be a random vector with support in $\Eos$ and whose margins all
  admit a logarithmic moment and a negative first moment:
  \[
    \esp\big[\vert \log Z^{j} \vert\big] < +\infty\ \text{and\
    } \esp\Big[\frac{1}{Z^{j}}\Big] < +\infty,\ j = 1,\dots,d.
  \]
  Then $\bm{Z}$ is a simple max-stable random vector with angular measure $\nu$
  if and only if
  \begin{align}\label{4_Stein_id}
    \esp\big[\langle \bm{Z}, \nabla f(\bm{Z}) \rangle\big] = \esp\big[\Dn f(\bm{Z})\big]
  \end{align}
  for all $f \in \cilog$.
\end{proposition}

\begin{proof}
  We start with the direct implication. Let $\bm{Z} \sim \mathcal{MS}(1, \nu)$
  and $\mu$ its exponent measure. We must show \eqref{4_Stein_id} for any $g$
  satisfying the assumptions of the theorem. Let $\eta$ be a Poisson process on
  $\Eo$ with intensity measure $\mu$. By Campbell-Mecke's formula (see
  \cite{Last17}) applied to the mapping
  \[
    \bm{y} \longmapsto \big\langle \bm{y},\nabla g\big(\m(\eta)\oplus \bm{y}\big)\big\rangle_{\m(\eta)}.
  \]
  and identity \eqref{4_Dn_alt}, we see that
  \begin{align*}
    \esp\big[\Dn g(\bm{Z})\big] = \esp\big[\Dn g\big(\m(\eta)\big)\big] & = \int_{\Eo} \esp\Big[\big\langle \bm{y}, \nabla g\big(\m(\eta)\oplus \bm{y}\big)\big\rangle_{\m(\eta)}\Big] \dif \mu(\bm{y})                   \\
                                                                        & = \esp\Big[\int_{\Eo} \big\langle \bm{y}, \nabla g\big(\m(\eta)\oplus \bm{y}\big)\big\rangle_{\m(\eta - \delta_{\bm{y}})}\dif \eta(\bm{y})\Big] \\
                                                                        & = \esp\Big[\int_{\Eo} \big\langle \bm{y}, \nabla g\big(\m(\eta)\big)\big\rangle_{\m(\eta - \delta_{\bm{y}})}\dif \eta(\bm{y})\Big]              \\
                                                                        & = \esp\big[\big\langle \m(\eta),\nabla g\big(\m(\eta)\big)\big\rangle\big] = \esp\big[\langle \bm{Z}, \nabla g(\bm{Z}) \rangle\big],
  \end{align*}
  giving us the announced identity. The penultimate equality comes from the fact
  that there is no point $\bm{y}$ in $\eta$ such that some coordinate of
  $\bm{y}$ is greater than the corresponding one of $\m(\eta)$. The last
  identity follows by observing that for every $j \in \lbra 1,d \rbra$, the only
  $\bm{y} \in \eta$ such that $y^{j}$ is greater than the $j$-th coordinate of
  $\m(\eta - \delta_{(r,\bm{u})})$ for some $j$ correspond to the ones giving
  the $j$-th coordinate of $\m(\eta)$.

  We turn to the reverse implication. The assumption on the marginals of
  $\bm{Z}$ ensures that both sides of \eqref{4_Stein_id} are finite. Notice that
  the latter identity reads also as $\esp[\Ln f(\bm{Z})] = 0$ for
  $f \in \cilog$. A dominated convergence argument then gives that
  \[
    \frac{\dif}{\dif t}\esp\big[\Pn_{t}f(\bm{Z})\big] = \esp\big[\Ln\Pn_{t}f(\bm{Z})\big] = 0
  \]
  thanks to \eqref{4_Stein_id}, because $\Pn_{t}f$ is log-Lipschitz due to
  theorem \ref{4_cilog_invar}. As a result, we deduce that for every $t \geq 0$:
  \[
    \esp\big[\Pn_{t}f(\bm{Z})\big] = \esp[f(\bm{Z})].
  \]
  As $(\Pn_{t})_{t\geq 0}$ is ergodic, its only invariant measure is
  $\prob_{1,\nu}$. This implies that $\bm{Z} \sim \mathcal{MS}(1,\nu)$.
\end{proof}

A pseudo Leibniz rule holds for $\Dn$:

\begin{proposition}[Pseudo Leibniz rule]
  For every $f,g \in \cilog$ and $x \in \Eos$:
  \begin{align}\label{4_pseudo_Leibniz}
    \Dn (fg)(\bm{x}) = \Dn f(\bm{x})g(\bm{x}) & + f(\bm{x})\Dn g(\bm{x})                                                                                            \\
                                              & + \int_{\Eo}\big(f(\bm{x} \oplus \bm{y})-f(\bm{x})\big)\big(g(\bm{x}\oplus \bm{y}) - g(\bm{x})\big)\dif\nu(\bm{y}).
  \end{align}
\end{proposition}

\begin{proof}
  Thanks to inequality \eqref{4_cilog_def}, we see that there exists $C > 0$
  such that:
  \begin{align*}
    \big\vert f(\bm{x} \oplus \bm{y})-f(\bm{x})\big\vert\big \vert g(\bm{x}\oplus \bm{y}) - g(\bm{x})\big\vert & \leq C\Vert \log (\bm{x}\oplus \bm{y}) - \log \bm{x} \Vert_{1}^{2} \\
                                                                                                               & \leq dC\sum_{j=1}^{d}(\log y^{j} - \log x^{j})^{2}_{+}
  \end{align*}
  The polar decomposition makes it clear that this last function in
  $\mu$-integrable over $\Eo$. Next, an easy computation yields:
  \begin{align*}
    \big(\Dn(fg) - f\Dn g - g\Dn f\big)(\bm{x}) & = \int_{\Eo}\big(f(\bm{x}\oplus \bm{y})g(\bm{x}\oplus \bm{y}) - f(\bm{x})g(\bm{x})\big) \dif\mu(\bm{y})                 \\
                                                & \quad - f(\bm{x})\int_{\Eo}\big(g(\bm{x}\oplus \bm{y}) - g(\bm{x})\big)\dif\mu(\bm{y})                                  \\
                                                & \quad - g(\bm{x})\int_{\Eo}\big(f(\bm{x}\oplus \bm{y}) - f(\bm{x})\big)\dif\mu(\bm{y})                                  \\
                                                & = \int_{\Eo}\big(f(\bm{x} \oplus \bm{y}) - f(\bm{x})\big)\big(g(\bm{x} \oplus \bm{y}) - g(\bm{x})\big) \dif\mu(\bm{y}).
  \end{align*}
\end{proof}

The \textit{carré du champ} operator associated to $(\Pn_{t})_{t\geq 0}$ is
denoted by
\[
  \Gamma_{1, \nu}(f, g) \coloneq \frac{1}{2}\big(\Ln(fg) - f\Ln g - g\Ln f\big),\ f, g \in \cilog.
\]
The corresponding \textit{Dirichlet form} is
\[
  \mathcal{E}_{1, \nu}(f, g) \coloneq \frac{1}{2}\esp\big[\Gamma_{1, \nu}(f, g)(\bm{Z}) \big],\ f, g \in \cilog.
\]
For more about the Bakry-Émery theory, we refer to \cite{Bakry14} and the
references therein.
\begin{lemma}
  We have for $\bm{x} \in \Eos$ and $f, g \in \cilog$:
  \[
    \Gamma_{1, \nu}(f, g)(\bm{x}) = \frac{1}{2}\int_{\Eo}\big(f(\bm{x}\oplus \bm{y}) - f(\bm{x})\big)\big(g(\bm{x}\oplus \bm{y}) - g(\bm{x})\big) \dif\mu(\bm{y}).
  \]
  Consequently, if $\bm{Z} \sim \mathcal{MS}(1, \nu)$, we have:
  \[
    \mathcal{E}_{1, \nu}(f) = \frac{1}{2}\int_{\Eo}\esp\big[\big(f(\bm{Z}\oplus \bm{y}) - f(\bm{Z})\big)^{2}\big] \dif\mu(\bm{y}).
  \]
\end{lemma}

\begin{proof}
  The purpose of the \textit{carré du champ} operator is to measure how far the
  generator $\Ln$ is from being a derivation, \textit{i.e.} from satisfying the
  Leibniz rule $(fg)' = fg' + gf'$. We know that
  \begin{align}\label{4_d+D}
    \Ln = \mathfrak{d}_{1, d} + \Dn,
  \end{align}
  where
  $\mathfrak{d}_{\alpha, d}f(\bm{x}) \coloneq -\alpha^{-1}\langle \bm{x}, \nabla f(\bm{x}) \rangle $
  is the generator of the $d$-dimensional dilation semi-group
  $(\mathfrak{p}_{t}^{\alpha, d})_{t\geq 0}$ defined by
  \[
    \mathfrak{p}_{t}^{\alpha, d}f(\bm{x}) \coloneq f\big(e^{-\frac{t}{\alpha}}\bm{x}\big),
  \]
  for every $\alpha \in \R^{*}$. One easily checks that
  $\mathfrak{d}_{\alpha, d}$ is a derivation and thus does not contribute to the
  \textit{carré du champ} operator:
  \begin{align*}
    2\Gamma_{1, \nu}(f, g)(\bm{x}) & = \big(\Ln(fg) - f\Ln g - g\Ln f\big)(\bm{x})                                                                       \\
                                   & = \big((\mathfrak{d}_{1, d} + \Dn)(fg) - f(\mathfrak{d}_{1, d} + \Dn)g - g(\mathfrak{d}_{1, d} + \Dn)f\big)(\bm{x}) \\
                                   & = \big(\Dn(fg) - f\Dn g - g\Dn f\big)(\bm{x}),
  \end{align*}
  which yields the result thanks to \eqref{4_pseudo_Leibniz}. The second
  identity stems from the fact that
  $\mathcal{E}_{1, \nu}(f) = \esp[\Gamma_{1, \nu}(f,f)(\bm{Z})]$.
\end{proof}

We deduce from the previous lemma and Poincaré inequality
\ref{3_Poincaré_max-id} for max-id random variables that $(\Pn_{t})t_{\geq 0}$
satisfies a Poincaré inequality with constant $2$. The exponential convergence
of the semi-group to its stationary measure is equivalent to the Poincaré
inequality, as exposed in \cite{Bakry14}.

\begin{proposition}
  Let $\bm{Z} \sim \mathcal{MS}(1, \nu)$ be a max-stable random vector, and
  $f \in \mathbf{L}^{2}(\pms_{1, \nu})$. Then we have:
  \begin{align*}
    \var\big(f(\bm{Z})\big) \leq 2\mathcal{E}_{1, \nu}(f).
  \end{align*}
  Thus $(\Pn_{t})_{t\geq 0}$ converges exponentially fast to its stationary
  measure $\pms_{1, \nu}$ in $\mathbf{L}^{2}(\pms_{1, \nu})$:
  \begin{align*}
    \Vert \Pn_{t}f - \esp[f(\bm{Z})] \Vert_{\mathbf{L}^{2}(\pms_{1, \nu})} \leq e^{-\frac{t}{2}}\Vert f - \esp[f(\bm{Z})] \Vert_{\mathbf{L}^{2}(\pms_{1, \nu})},
  \end{align*}
  for all $t \geq 0$ and $f \in \cilog$.
\end{proposition}

\subsection{The general case}\label{4_3}

We wish to extend the definition of $\Pn_{t}$ to arbitrary max-stable
distributions. Let us introduce a few notations. First, for
$\Psi : \R^{d} \to \R^{d}$, denote by $T_{\Psi}$ the map defined by
$T_{\Psi}f \coloneq f\circ \Psi$ for all $f:\R^{d}\to \R$. It is clear that if
$\Psi$ is invertible, then $T_{\Psi}^{-1} = T_{\Psi^{-1}}$. Define
\begin{align*}
  \begin{array}{ccccc}
    \psi_{\alpha} & : & x & \longmapsto & \begin{cases}
                                            x^{\alpha} \quad    & \text{if\ } \alpha > 0  \\
                                            \exp x \quad        & \text{if\ } \alpha = 0  \\
                                            (-x)^{\alpha} \quad & \text{if\ } \alpha < 0.
                                          \end{cases}
  \end{array}
\end{align*}
and for $\bm{\alpha} \in \R^{d}$:
\begin{align*}
  \begin{array}{ccccc}
    \Psi_{\bm{\alpha}} & : & \bm{x} & \longmapsto & \big(\psi_{\alpha^{1}}(x^{1}),\dots,\psi_{\alpha^{d}}(x^{d})\big).
  \end{array}
\end{align*}
his transformation is clearly bijective, as well as non-decreasing with respect
to each coordinate. If $\bm{\alpha} = \alpha\bm{1}$ for some $\alpha \in \R$, we
note $\Psi_{\alpha}$ instead of $\Psi_{\alpha\bm{1}}$ for short. We also set
$T_{\alpha} \coloneq T_{\Psi_{\alpha}}$.

With those notations, a basic result in extreme-value theory (proposition 5.10.
in \cite{Resnick87}) can be stated as follows: if $\bm{Z}$ is a max-stable
random vector, then there exists a unique $\bm{\alpha} \in \R^{d}$ such that:
\[
  \Psi_{\bm{\alpha}}(\bm{Z}) \eqdis \mathcal{MS}(1, \nu).
\]
for some angular measure $\nu$. The distribution of $\bm{Z}$ will be denoted by
$\mathcal{MS}(\bm{\alpha}, \nu)$. This notation is consistent with the one
introduced in the preliminaries for simple max-stable random vectors.

An essential property of $T_{\bm{\alpha}}$ is the following.
\begin{proposition}
  For every $\bm{\alpha} \in \R^{d}$ and $p \in [1, +\infty]$, the application
  $T_{\bm{\alpha}}$ is an isometry from $\L^{p}(\pms_{\alpha, \nu})$ to
  $\L^{p}(\pms_{1, \nu})$:
  \[
    \Vert T_{\bm{\alpha}}f \Vert_{\L^{p}(\pms_{\alpha, \nu})} = \Vert f \Vert_{\L^{p}(\pms_{1, \nu})},
  \]
  for every $f \in \L^{p}(\pms_{1, \nu})$.
\end{proposition}

Using this application, we can extend the definition of the max-stable
Ornstein-Uhlenbeck to every max-stable random vector.


\begin{definition}[Generalized max-stable Ornstein-Uhlenbeck semi-group]
  Let $\bm{\alpha}$ belong to $\R^{d}$. The \textit{generalized max-stable
    Ornstein-Uhlenbeck semi group} $(\Panbm_{t})_{t\geq 0}$ is defined on
  $\L^{p}(\pms_{\bm{\alpha}, \nu})$ for $p \in [1, +\infty]$ by setting
  \begin{align}\label{4_TPTinv}
    \Panbm_{t} \coloneq T_{\bm{\alpha}}\Pn_{t}T_{\bm{\alpha}}^{-1},\ t\geq 0.
  \end{align}
\end{definition}

With this definition, it is easy to check that $(\Panbm_{t})_{t\geq 0}$ is a
Markov semi-group. Using the isometry property of $T_{\bm{\alpha}}$, one finds
the generator $\Lanbm$ of this semi-group.
\begin{proposition}
  For every $\bm{\alpha} \in \R^{d}$ and $f \in T_{\bm{\alpha}}^{-1}\cilog$, we
  have
  \[
    \Lanbm f(\bm{x}) = \big(T_{\bm{\alpha}}\mathfrak{d}_{1}T_{\bm{\alpha}}^{-1}\big) + \Danbm f(\bm{x}),\ \bm{x} \in T_{\bm{\alpha}}^{-1}\Eos.
  \]
  where $\mathfrak{d}_{1}f(\bm{x}) = -\langle \bm{x}, \nabla f(\bm{x}) \rangle$,
  and $\Danbm \coloneq T_{\bm{\alpha}}\Dn T_{\bm{\alpha}}^{-1}$. Let $\mu$ be
  the exponent measure of $\bm{Z} \sim \mathcal{MS}(\bm{\alpha}, \nu)$, the
  latter having support in $\El$ for some
  $\bm{\ell} \in [-\bm{\infty}, +\bm{\infty})$, then one has:
  \[
    \Danbm f(\bm{x}) = \int_{\El} \big(f(\bm{x}\oplus \bm{y}) - f(\bm{x})\big) \dif\mu(\bm{y}).
  \]
\end{proposition}

We explicit the expression of $\Panbm_{t}f$ in the case
$\bm{\alpha} = \alpha\bm{1}$ for some $\alpha \in \R$. In that case, we write
$\Pan_{t}$ instead of $\Panbm_{t}$. This notation is consistent with the one we
used for the standard MSOU semi-group $(\Pn_{t})_{t\geq 0}$.

\begin{example} Let $\alpha \in \R$ and $\bm{Z} \sim \mathcal{MS}(\alpha, \nu)$.
  There are three cases.

  1. $\alpha > 0$: The marginals of $\bm{Z}$ are all Fréchet
  $\mathcal{F}(\alpha)$ and $\Psi_{\alpha}^{-1}\Eos = \Eos$. Also,
  $T_{\alpha}^{-1}\cilog = \cilog$ and since
  $\bm{Z}^{\alpha} \sim \mathcal{MS}(1, \nu)$, one has
  \begin{align*}
    \Pan_{t}f(\bm{x}) & = \big(T_{\alpha}\Pn T_{\alpha}^{-1}\big)f(\bm{x})                                                        \\
                      & = \big(\Pn T_{\alpha}^{-1}\big)f(\bm{x}^{\alpha})                                                         \\
                      & = \esp\big[\big(T_{\alpha}^{-1}f\big)\big(e^{-t}\bm{x}^{\alpha}\oplus (1-e^{-t})\bm{Z}^{\alpha}\big)\big] \\
                      & = \esp\big[f\big(e^{-\frac{t}{\alpha}}\bm{x} \oplus (1-e^{-t})^{\frac{1}{\alpha}}\bm{Z}\big)\big].
  \end{align*}
  The generator $\Lan$ of this semi-group is given by
  \begin{align*}\label{4_Lan}
    \Lan f(\bm{x}) & = -\frac{1}{\alpha}\langle \bm{x},\nabla f(\bm{x}) \rangle + \frac{1}{\alpha}\int_{\Epol}\big\langle r\bm{u}^{1/\alpha}, \nabla f(\bm{x}\oplus r\bm{u}^{1/\alpha}) \big\rangle_{\bm{x}} \frac{\alpha}{r^{\alpha + 1}}\dif r \dif\nu(\bm{u})                               \\
                   & = -\frac{1}{\alpha}\langle \bm{x},\nabla f(\bm{x}) \rangle + \frac{1}{\alpha}\int_{(\S_{+}^{d-1})^{1/\alpha}}\int_{\Rps}\big\langle r\bm{v}, \nabla f(\bm{x}\oplus r\bm{v}) \big\rangle_{\bm{x}} \frac{\alpha}{r^{\alpha + 1}}\dif r \dif\nu_{\alpha}(\bm{v}),\numberthis
  \end{align*}
  where $\nu_{\alpha}$ is the pushforward measure of $\nu$ by
  $\Psi_{\alpha}(\bm{x}) = \bm{x}^{\alpha}$ and $(\S_{+}^{d-1})^{1/\alpha}$ the
  set of elements of the form $\bm{v} = \bm{u}^{1/\alpha}$ for some
  $\bm{u} \in \S_{+}^{d-1}$.

  2. $\alpha = 0$: The marginals of $\bm{Z}$ are all standard Gumbel
  $\mathcal{G}(0,1)$, with c.d.f. $x \mapsto \exp(-\exp(-x))$ on $\R$. We see
  that $\Psi_{0}^{-1}\Eos = E^{*}_{-\bm{\infty}} = \R^{d}$ and
  $T_{0}^{-1}\cilog = \mathcal{C}_{\text{Lip}}^{1}(\R^{d})$, the space of class
  $\mathcal{C}^{1}$ Lipschitz functions on $\R^{d}$. The semi-group
  $(\mathbf{P}^{0, \nu}_{t})_{t\geq 0}$ can be expressed as
  \[
    \mathbf{P}^{0, \nu}_{t}f(\bm{x}) = \esp\big[f\big((\bm{x} - t) \oplus (\bm{Z} + \log(1-e^{-t}))\big)\big]
  \]
  and its generator equals
  \begin{align*}
    \mathscr{L}_{0, \nu}f(\bm{x}) & = -\langle \bm{1}, \nabla f(\bm{x}) \rangle + \int_{ \S_{+}^{d-1}}\int_{\R}\big\langle (r\bm{1} + \log \bm{v}), \nabla f(\bm{x}\oplus (r\bm{1} + \log \bm{v})) \big\rangle_{\bm{x}} e^{-r}\dif r \dif\nu(\bm{v}) \\
                                  & = -\langle \bm{1}, \nabla f(\bm{x}) \rangle + \int_{\log \S_{+}^{d-1}}\int_{\R}\big\langle (r\bm{1} + \bm{v}), \nabla f(\bm{x}\oplus (r\bm{1} + \bm{v})) \big\rangle_{\bm{x}} e^{-r}\dif r \dif\nu_{0}(\bm{v}),
  \end{align*}
  where $\nu_{0}$ is the pushforward measure of $\nu$ by $\Psi_{0} = \exp$ and
  $\log \S_{+}^{d-1}$ the set of elements of the form $\bm{v} = \log \bm{u}$ for
  some $\bm{u} \in \S_{+}^{d-1}$.

  3. $\alpha < 0$: The marginals of $\bm{Z}$ are all negative Weilbull
  $\mathcal{W}(\alpha)$ with c.d.f. $x \mapsto \exp((-x)^{-\alpha})$ on
  $\R_{-}$, so that $\Psi_{\alpha}^{-1}\Eos = \R_{-}^{d}$ and
  $T_{\alpha}^{-1}\cilog = \mathcal{C}_{\text{log}}^{1}(\R_{-}^{d})$. The
  semi-group takes the form
  \begin{align*}
    \Pan_{t}f(\bm{x}) & = \esp\big[f\big(e^{-\frac{t}{\alpha}}\bm{x} \oplus (1-e^{-t})^{\frac{1}{\alpha}}\bm{Z}\big)\big]
  \end{align*}
  while its generator $\Lan$ is
  \begin{align*}
    \Lan f(\bm{x}) & = -\frac{1}{\alpha}\langle \bm{x},\nabla f(\bm{x}) \rangle - \frac{1}{\alpha}\int_{\S_{+}^{d-1}}\int_{\R_{-}}\big\langle r\bm{u}^{1/\alpha}, \nabla f(\bm{x}\oplus (-r\bm{u}^{1/\alpha})) \big\rangle_{\bm{x}} \frac{\alpha}{(-r)^{\alpha + 1}}\dif r \dif\nu(\bm{u}) \\
                   & = -\frac{1}{\alpha}\langle \bm{x},\nabla f(\bm{x}) \rangle - \frac{1}{\alpha}\int_{-(\S_{+}^{d-1})^{1/\alpha}}\int_{\R_{-}}\big\langle r\bm{v}, \nabla f(\bm{x}\oplus r\bm{v}) \big\rangle_{\bm{x}} \frac{\alpha}{(-r)^{\alpha + 1}}\dif r \dif\nu_{\alpha}(\bm{v}).
  \end{align*}
  Those expressions are formally the same as in the case $\alpha > 0$, although
  the definition sets are different.

\end{example}

The fact that $\mathcal{MS}(\bm{\alpha}, \nu)$ is an invariant measure of
$(\Panbm_{t})_{t\geq 0}$, as well as the ergodicity of this semi-group are other
easy consequences of \eqref{4_TPTinv} and the properties of
$(\Pn_{t})_{t\geq 0}$. The operator $\Dan$ also satisfies the commutation rule:
\begin{align*}
  \Danbm \Panbm & = \big(T_{\bm{\alpha}}\Dn T_{\bm{\alpha}}^{-1}\big)\big(T_{\bm{\alpha}}\Pn_{t} T_{\bm{\alpha}}^{-1} \big) \\
                & = T_{\bm{\alpha}}\Dn \Pn_{t} T_{\bm{\alpha}}^{-1}                                                         \\
                & = e^{-t}T_{\bm{\alpha}}\Pn_{t} \Dn T_{\bm{\alpha}}^{-1}                                                   \\
                & = e^{-t}\Panbm_{t} \Danbm.
\end{align*}
Notice that $\bm{\alpha}$ does not appear in the exponential. Equation
\eqref{4_inf_comm_Frcht_d} holds true as well:
\[ [\Lanbm, \Danbm] = \Danbm.
\]
Likewise, one can easily retrieve the pseudo-Leibniz rule as well as the
Poincaré inequality stated at the end of the previous section for
$(\Pn_{t})_{t\geq 0}$.

We conclude this section by noticing that one can extend the previous
construction in at least two directions. First, the min-stable distributions:
this amounts to replacing $\psi_{\alpha}$ by $\psi_{\alpha}(x^{-1})$ if
$\alpha \neq 0$, or by $\psi_{0}(-x)$ otherwise. For example, since the
exponential distribution with unit parameter $\mathcal{E}(1)$ is min-stable, a
Markov semi-group admitting this law as its invariant measure is given by:
\[
  \mathbf{P}_{t}f(x) = \esp\big[f\big(e^{t}x \odot (1-e^{-t})^{-1}Z\big)\big],\ x \geq 0
\]
where $Z \sim \mathcal{E}(1)$.

Second, one can try and apply those ideas to max-id distributions. In dimension
$1$, this case is the most general possible: any random variable $Z$ is max-id.
Assume for simplicity that $F_{Z}$ is invertible and take
$\psi(x) = -1/\log F_{Z}(x)$. This function is defined on the support of $Z$ and
one has
\[
  \psi(Z) \sim \mathcal{F}(1).
\]
Consequently, the operators
\[
  \mathbf{P}_{t}f(x) = T_{\psi}\Pn_{t}T_{\psi}^{-1}
\]
form a Markov semi-group whose stationary measure is the distribution of $Z$.
For instance, if $Z\sim \mathcal{U}[0,1]$, then $\mathbf{P}_{t}f$ takes the
following form:
\[
  \mathbf{P}_{t}f(x) = \esp\Big[f\big(x^{e^{t}} \oplus U^{\frac{1}{1-e^{-t}}}\big)\Big],\ x \in [0,1]
\]
where $U \sim \mathcal{U}[0,1]$. A more convoluted expression arises for the
logistic distribution with c.d.f. $(1+e^{-x})^{-1}$ on $\R$:
\[
  \mathbf{P}_{t}f(x) = \esp\Big[f\Big(-\log\big((1+e^{-x})^{t}-1\big) \oplus -\log\big(e^{\frac{1}{(1-e^{-t})Z}} - 1\big)\Big)\Big],\ x \in \R
\]
since $-\log(e^{1/Z}-1)$ has the logistic distribution if $Z$ has the unit
Fréchet distribution. This time we cannot give a Mehler formula for
$\mathbf{P}_{t}f(x)$ by using a random variable having the target logistic
distribution. This stems from the fact that the distribution of a maximum of two
\iid logistic random variables is not easily expressed in terms of one logistic
distribution. In higher dimensions, things become even more difficult, as not
every max-id distribution can be realized as a monotone function of a max-stable
random vector. A possibility is then to restrict one's attention to \textit{self
  max-decomposable distributions}, paralleling the approach of Arras and Houdré
in \cite{Arras19}. This path is currently being investigated by the authors.

\subsection{Specialization to the univariate case}

In this subsection we focus on the case $d = 1$ and assume $\alpha > 0$. The
univariate case when $\alpha = 0$ is studied in \cite{Costaceque24_cpn} and
applied to the coupon collector problem. The case of the negative Weibull
distribution ($\alpha < 0$) is formally similar to the one studied in this
subsection, although with heavier notations due to omnipresence of minus signs.

In the univariate case and with our choice of normalization, the only possible
angular measure is the Dirac mass at $1$, so we will note $\Pa_{t}$ instead of
$\Pan_{t}$. The same goes for the associated operators $\La$ and $\Da$. Recall
that $\gamma_{t} = e^{t} - 1$.

Inspired by the classic identity \eqref{4_Bakry}, we have proved a commutation
relation between $\Da$ and $\Pa_{t}$. When replacing $\Da$ by the gradient, we
find instead the next result.

\begin{proposition}\label{4_Frcht_Pt}
  Let $f$ be a $\mathcal{C}^{1}(\Rps)$-class function, such that $f$ and $f'$
  are integrable on $\Rps$ with respect to Lebesgue measure. Then we have the
  following: $\Pa_{t}f$ is differentiable and:
  \begin{align}\label{4_Frcht_deriv}
    (\Pa_{t}f)'(x) = e^{-\frac{t}{\alpha}}e^{-\frac{\gamma_{t}}{x^{\alpha}}}f'\big(e^{-\frac{t}{\alpha}}x\big)
  \end{align}
  Consequently $\Pa_{t}f$ satisfies:
  \begin{align}\label{4_Frcht_alt}
    \Pa_{t}f(x) = -e^{-\frac{t}{\alpha}}\int_{x}^{\infty}e^{-\frac{\gamma_{t}}{r^{\alpha}}}f'\big(e^{-\frac{t}{\alpha}}r\big)\dif r.
  \end{align}
\end{proposition}

\begin{proof}

  Because $f$ is bounded on $[x, +\infty]$ for every $x > 0$, we see that
  $f \in \L^{1}(\pms_{1})$, so that $\Pa_{t}f$ is well-defined.

  Let $x > 0$. Conditioning on whether $Z \leq \gamma_{t}^{1/\alpha}x$ or not,
  we find the decomposition
  \begin{align*}\label{4_cases_Frcht}
    \Pa_{t}f(x) & = \esp\big[f\big(e^{-\frac{t}{\alpha}}x\big)\ind_{\lb Z \leq \gamma_{t}^{1/\alpha}x \rb}\big] + \esp\big[f\big((1-e^{-t})^{\frac{1}{\alpha}}Z\big)\ind_{\lb Z > \gamma_{t}^{1/\alpha}x \rb}\big]                       \\
                & = f\big(e^{-\frac{t}{\alpha}}x\big)e^{-\frac{\gamma_{t}}{x^{\alpha}}} + \gamma_{t}\int_{x}^{\infty}f\big(e^{-\frac{t}{\alpha}}z\big)e^{-\frac{\gamma_{t}}{z^{\alpha}}}\frac{\alpha}{z^{\alpha + 1}}\dif z. \numberthis
  \end{align*}
  To prove equation \eqref{4_Frcht_deriv}, we simply differentiate that last
  expression with respect to $x$
  \begin{align*}
    (\Pa_{t}f)'(x) & = \frac{\dif}{\dif x}\Big(f\big(e^{-\frac{t}{\alpha}}x\big)e^{-\frac{\gamma_{t}}{x^{\alpha}}} + \gamma_{t}\int_{x}^{\infty}f\big(e^{-\frac{t}{\alpha}}z\big)e^{-\frac{\gamma_{t}}{z^{\alpha}}}\frac{\alpha}{z^{\alpha + 1}}\dif z\Big)                                                                                \\
                   & = e^{-\frac{t}{\alpha}}f'\big(e^{-\frac{t}{\alpha}}x\big)e^{-\frac{\gamma_{t}}{x^{\alpha}}} + \gamma_{t}f\big(e^{-\frac{t}{\alpha}}x\big)e^{-\frac{\gamma_{t}}{x^{\alpha}}}\frac{\alpha}{x^{\alpha + 1}} - \gamma_{t}f\big(e^{-\frac{t}{\alpha}}x\big)e^{-\frac{\gamma_{t}}{x^{\alpha}}}\frac{\alpha}{x^{\alpha + 1}} \\
                   & = e^{-\frac{t}{\alpha}}e^{-\frac{\gamma_{t}}{x^{\alpha}}}f'\big(e^{-\frac{t}{\alpha}}x\big).
  \end{align*}

  The integral in the right-hand side of \eqref{4_Frcht_alt} exists because $f'$
  is integrable on $[x, +\infty)$ and for any $x > 0$. Identity
  \eqref{4_Frcht_alt} is a direct consequence of the previous display and of the
  fundamental theorem of calculus.

\end{proof}

\begin{proposition}\label{4_Frcht_gnrt}
  Let $f \in \cilogi $ and $x \in \Rps$ and $Z \sim \mathcal{F}(\alpha)$. Let
  $Y$ have the Pareto distribution $\mathcal{F}(\alpha)$, with density:
  \[
    x \mapsto \frac{\alpha}{x^{\alpha + 1}}\ind_{[1, +\infty)}(x).
  \]
  The generator $\La$ of the univariate MSOU semi-group satisfies:
  \begin{align}
    \La f(x) & = -\frac{1}{\alpha}xf'(x) + \frac{1}{x^{\alpha}}\esp\big[f(xY) - f(x)\big] \label{4_Frcht_gnrtr_alt1}      \\
             & = -\frac{1}{\alpha}xf'(x) + \frac{1}{\alpha}\frac{1}{x^{\alpha-1}}\esp[Yf'(xY)]. \label{4_Frcht_gnrt_alt2}
  \end{align}
\end{proposition}

\begin{proof}
  Recall that $\rho_{\alpha}$ has been defined at \eqref{prelim_rho}. The change
  of variable $u = r/x$ gives immediately:
  \[
    \Da f(x) = \int_{x}^{\infty}\big(f(r) - f(x)\big)\frac{\alpha}{r^{\alpha + 1}}\dif r = \frac{1}{x^{\alpha}}\int_{1}^{\infty}\big(f(xu) - f(x)\big) \dif\rho_{\alpha}(u).
  \]
  Identifying the Pareto distribution $\mathcal{VP}(\alpha)$, we obtain the
  announced result just as easily. Next, identity \eqref{4_Dn_alt} yields
  \[
    \Da f(x) = \int_{x}^{\infty}f'(r)\frac{1}{r^{\alpha}}\dif r = \frac{1}{\alpha}\int_{x}^{\infty}rf'(r)\dif \rho_{\alpha}(r).
  \]
\end{proof}

We now give two covariance identities for $(\Pa_{t})_{t\geq 0}$.

\begin{proposition}[Covariance identities]
  Let $f, g \in \cilogi$ and $Z \sim \mathcal{F}(\alpha)$.
  \begin{enumerate}

    \item Let $Y \sim \mathcal{VP}(\alpha)$ be a random variable with Pareto
          distribution, independent of $Z$. Then:
          \begin{align}\label{4_Frcht_cov_id1}
            \langle \La f,g \rangle_{\mathbf{L}^{2}(\pms_{\alpha})} = -\frac{1}{\alpha^{2}}\esp\big[YZ^{2}f'(YZ)g'(Z)\big].
          \end{align}

    \item Assume further that $f$ has zero mean: $\esp[f(Z)] = 0$. Then:
          \begin{align}\label{4_Frcht_cov_id2}
            \langle f,g \rangle_{\mathbf{L}^{2}(\pms_{\alpha})} = -\frac{1}{\alpha^{2}}\esp\big[YZ^{2}(\La^{-1}f)'(YZ)g'(Z)\big].
          \end{align}

  \end{enumerate}
\end{proposition}

\begin{proof}

  1. Integrating the density of the Fréchet distribution and differentiating the
  rest in the second term below, one finds:
  \begin{align*}
    \langle \La f,g \rangle_{\mathbf{L}^{2}(\pms_{\alpha})} & = -\frac{1}{\alpha}\langle x\nabla f,g \rangle_{\mathbf{L}^{2}(\pms_{\alpha})} + \big\langle \Da f, g \big\rangle_{\mathbf{L}^{2}(\pms_{\alpha})}                                                                               \\
                                                            & = -\frac{1}{\alpha}\langle x\nabla f,g \rangle_{\mathbf{L}^{2}(\pms_{\alpha})} + \int_{0}^{\infty} \Big(\int_{x}^{\infty}f'(y)\frac{1}{y^{\alpha}}\dif y\Big) g(x) \frac{\alpha}{x^{\alpha + 1}}e^{-\frac{1}{x^{\alpha}}}\dif x \\
                                                            & = -\int_{0}^{\infty} \Big(\int_{1}^{\infty}\frac{1}{y^{\alpha}} f'(xy)\dif y\Big) g'(x) \frac{1}{x^{\alpha-1}}e^{-\frac{1}{x^{\alpha}}}\dif x                                                                                   \\
                                                            & = -\frac{1}{\alpha^{2}}\int_{0}^{\infty} \Big(\int_{1}^{\infty}yf'(xy) \frac{\alpha}{y^{\alpha + 1}}\dif y\Big) x^{2}g'(x) \frac{\alpha}{x^{\alpha + 1}} e^{-\frac{1}{x^{\alpha}}}\dif x                                        \\
                                                            & = -\frac{1}{\alpha^{2}}\esp\big[YZ^{2}f'(YZ)g'(Z)\big].
  \end{align*}
  We have used the change of variable $y' = y/x$ to obtain the fourth identity.

  2. This relation is a direct consequence of the first, by replacing $f$ by
  $\La^{-1}f$. However the latter does not belong to $\cilogi$, so it is not
  obvious that the right-hand side of \eqref{4_Frcht_cov_id2} makes sense. By
  adapting the proof of proposition \ref{4_Ln-1}, we find that there exists some
  $C > 0$:
  \[
    x\vert (\La^{-1}f)'(x) \vert \leq C\int_{0}^{\infty}e^{-\gamma_{t}x^{-\alpha}} \dif t.
  \]
  Assume $C = 1$ for ease of notations. From that inequality, one deduces
  \begin{align*}
    \esp\big[YZ\vert(\La^{-1}f)'(YZ)\vert\big] & \leq \int_{0}^{\infty}\esp\big[e^{-\gamma_{t}(YZ)^{-\alpha}}\big] \dif t           \\
                                               & = \int_{0}^{\infty}\esp\Big[\frac{Y^{\alpha}}{\gamma_{t} + Y^{\alpha}}\Big] \dif t \\
                                               & = \esp\Big[\int_{0}^{\infty}\frac{Y^{\alpha}}{e^{t} + Y^{\alpha} - 1}\Big] \dif t  \\
                                               & = \esp\Big[\frac{Y^{\alpha}}{Y^{\alpha} - 1}\log Y^{\alpha}\Big]                   \\
                                               & = \int_{0}^{\infty}\frac{\log (y+1)}{y(y+1)} \dif y = \frac{\pi^{2}}{6}.
  \end{align*}
  In particular, $YZ(\La^{-1}f)'(YZ)$ is integrable, and thus so is
  $YZ^{2}(\La^{-1}f)'(YZ)g'(Z)$, since $Zg'(Z)$ is bounded. Using the same
  arguments as in the previous point, one proves that \eqref{4_Frcht_cov_id1}
  remains valid when $f$ is replaced by $\La^{-1}f$, thus concluding the proof.
\end{proof}

It is well-known (\textit{e.g.} \cite{Nourdin12}) that the generator
$\mathcal{L}$ of the univariate Ornstein-Uhlenbeck semi-group satisfies
\begin{align}\label{4_deltaD_OU}
  \mathcal{L}f(x) = -xf'(x) + f''(x) = (\delta\circ\nabla) f(x),
\end{align}
where $\delta \coloneq -x + \nabla$ is known as the divergence operator. It is
equal to the adjoint of the usual derivative operator $\nabla$ with respect to
the scalar product $\langle f,g \rangle_{\mathbf{L}^{2}(\R, \gamma)}$ and
$\gamma$ denotes the standard normal distribution $\mathcal{N}(0,1)$. In
particular $\mathcal{L}$ is self-adjoint. As we already noticed, our generator
$\La$ does not share this property, but it nonetheless satisfies a similar
relation:
\begin{align}\label{4_deltaD_OUMS}
  \La f(x) = -\frac{1}{\alpha}xf'(x) + \frac{1}{\alpha}\int_{x}^{\infty}rf'(r)\dif \rho_{\alpha}(r) = (\delta_{\alpha} \circ \Da)f(x),
\end{align}
where the operator $\delta_{\alpha}$ is equal for $f\in \cilogi$ to
\[
  \delta_{\alpha}f(x) \coloneq (\alpha^{-1}x^{\alpha+1}\nabla + \mathrm{Id})f(x) = \alpha^{-1}x^{\alpha+1}f'(x) + f(x).
\]
This operator is actually a Stein operator, as proved in \cite{Bartholome13,
  Kusumoto20}, where it is denoted by $\mathcal{T}_{\alpha}$, up to a constant
$\alpha^{-1}$. Equality \eqref{4_deltaD_OUMS} makes a connection between their
operator and $\La$. The divergence $\delta_{\alpha}$ satisfies several
properties. The first one originates from \cite{Bartholome13}.

\begin{proposition}[Integration-by-parts formula]
  Let $f, g$ be as in theorem \ref{4_Frcht_Pt}, \textit{i.e.} of class
  $\mathcal{C}^{1}$, integrable and with first derivative integrable as well.
  Then we have
  \begin{align}\label{4_IPP_Swan}
    \langle \delta_{\alpha}f,g \rangle_{\mathbf{L}^{2}(\pms_{\alpha})} = -\big\langle f, \alpha^{-1}r^{\alpha + 1}\nabla g \big\rangle_{\mathbf{L}^{2}(\pms_{\alpha})}
  \end{align}
\end{proposition}

Just like $\Da$, the operator $\delta_{\alpha}$ satisfies a commutation relation
with $\Pa_{t}$.

\begin{proposition}[Commutation relation for $\delta_{\alpha}$ - Fréchet case]
  Let $f$ be as in theorem \ref{4_Frcht_Pt}. Then we have:
  \begin{align}\label{4_cmtn_rltn2}
    \delta_{\alpha}\Pa_{t}f(x) = e^{t}\Pa_{t}\delta_{\alpha}f(x),\ x \in \Rps,\ t\geq 0.
  \end{align}
\end{proposition}

\begin{proof}
  We compute each side of the equality, starting with $\delta_{\alpha}\Pa_{t}f$:
  \begin{align*}
    \delta_{\alpha}\Pa_{t}f(x) & = \alpha^{-1}x^{\alpha + 1}f'\big(e^{-\frac{t}{\alpha}}x\big)e^{-\frac{t}{\alpha}}e^{-\frac{\gamma_{t}}{x^{\alpha}}} + \Pa_{t}f(x),
  \end{align*}
  thanks to identity \eqref{4_Frcht_deriv}, while the second part is equal to:
  \begin{align*}
    \Pa_{t}\delta_{\alpha}f(x) & = \alpha^{-1}\Pa_{t}(x^{\alpha + 1}\nabla f)(x) + \Pa_{t}f(x).
  \end{align*}
  Using decomposition \eqref{4_cases_Frcht}, one finds:
  \begin{align*}
    \Pa_{t}\delta_{\alpha}f(x) & = \Pa_{t}f(x) + \alpha^{-1}\Pa_{t}(x^{\alpha + 1}\nabla f)(x)                                                                                                                                                                                                                  \\
                               & = \Pa_{t}f(x) + e^{-(\alpha + 1)\frac{t}{\alpha}}\Big(\alpha^{-1}x^{\alpha + 1}f'\big(e^{-\frac{t}{\alpha}}x\big)e^{-\frac{\gamma_{t}}{x^{\alpha}}} + \alpha^{-1}\gamma_{t}\int_{x}^{\infty}f'\big(e^{-\frac{t}{\alpha}}z\big)e^{-\frac{\gamma_{t}}{z^{\alpha}}}\dif z(x)\Big) \\
                               & = \Pa_{t}f(x) + e^{-(\alpha + 1)\frac{t}{\alpha}}\Big(\alpha^{-1}x^{\alpha + 1}f'\big(e^{-\frac{t}{\alpha}}x\big)e^{-\frac{\gamma_{t}}{x^{\alpha}}} - \gamma_{t}e^{\frac{t}{\alpha}}\Pa_{t}f(x)\Big)                                                                           \\
                               & = \Pa_{t}f(x) + (e^{-t}-1)\Pa_{t}f(x) + e^{-(\alpha + 1)\frac{t}{\alpha}}\alpha^{-1}x^{\alpha + 1}f'\big(e^{-\frac{t}{\alpha}}x\big)e^{-\frac{\gamma_{t}}{x^{\alpha}}}                                                                                                         \\
                               & = e^{-t}\delta_{\alpha}\Pa_{t}f(x).
  \end{align*}
  We have used \eqref{4_Frcht_alt} to recognize $\Pa_{t}f$ at the third line.
\end{proof}


Denote by $[A, B] \coloneq A\circ B - B\circ A$ the commutator between two
endomorphisms of $\cilogi $. It serves as a tool to measure the lack of
commutativity between $A$ and $B$ since $[A, B] = 0$ (the null operator) if and
only if $A$ and $B$ commute. The commutator plays a fundamental role in quantum
mechanics, see for instance \cite{Hall13}. The next identities show that $\La$,
$\Da$, $\delta_{\alpha}$ and $\mathrm{Id}$ span a Lie algebra.

\begin{proposition}[Commutator identities]
  For the functions satisfying the assumptions of theorem \ref{4_Frcht_Pt}, we
  have the following relations:
  \begin{align}\label{4_C-A} [\delta_{\alpha}, \Da] = \mathrm{Id}.
  \end{align}
  \begin{align}\label{4_inf_comm} [\La, \Da] = \Da.
  \end{align}
  \begin{align}\label{4_thrd_comm} [\delta_{\alpha}, \La] = \delta_{\alpha}.
  \end{align}
  Furthermore, $\La$, $\Da$ and $\delta_{\alpha}$ satisfy the Jacobi identity:
  \begin{align}\label{4_Jcbi}
    \big[\La, [\Da, \delta_{\alpha}]\big] + \big[\Da, [\delta_{\alpha}, \La]\big] + \big[\delta_{\alpha}, [\La, \Da]\big] = 0.
  \end{align}
\end{proposition}

\begin{proof}
  We make use of equality \eqref{4_Frcht_gnrt_alt2}. Notice we can ignore the
  identity part in
  $\delta_{\alpha} = \alpha^{-1}x^{\alpha+1}\nabla + \mathrm{Id}$ since it
  commutes with $\Da$. Let $f\in \cilogi $ and $x \in \Rps$.
  \begin{align*}
    \alpha[\delta_{\alpha}, \Da]f(x) & = [x^{\alpha + 1}\nabla, \Da]f(x)                                                                                        \\
                                     & = x^{\alpha + 1}\nabla(\Da f)(x) - \Da(x^{\alpha + 1}f')(x)                                                              \\
                                     & = x^{\alpha + 1}\big(-x^{-\alpha}f'(x)\big) - \int_{x}^{\infty}\big(r^{\alpha + 1}f'(r)\big)'\frac{1}{r^{\alpha}} \dif r \\
                                     & = -xf'(x) + xf'(x) - \alpha\int_{x}^{\infty}f'(r)\dif r                                                                  \\
                                     & = \alpha f(x).
  \end{align*}
  The proof of the second identity is rather similar:
  \begin{align*}
    [\La, \Da]f(x) & = -\frac{1}{\alpha}[x\nabla, \Da]f(x)                                                                          \\
                   & = -\frac{1}{\alpha}x(\Da f)'(x) + \frac{1}{\alpha}\Da (xf')(x)                                                 \\
                   & = -\frac{1}{\alpha}x^{-\alpha} + \frac{1}{\alpha}\int_{x}^{\infty}\big(rf'(r)\big)'\frac{1}{r^{\alpha}} \dif r \\
                   & = \int_{x}^{\infty}f'(r)\frac{1}{r^{\alpha}} \dif r.
  \end{align*}
  The final identity is not much harder to prove thanks to the first relation:
  \begin{align*}
    \alpha[\delta_{\alpha}, \La]f(x) & = -\frac{1}{\alpha}[x^{\alpha + 1}\nabla, x\nabla]f(x) + [x^{\alpha + 1}\nabla, \Da]f(x)                    \\
                                     & = \alpha f(x) + \frac{1}{\alpha}[x\nabla, x^{\alpha + 1}\nabla]f(x)                                         \\
                                     & = \alpha f(x) + \frac{1}{\alpha}\Big(x\big(x^{\alpha + 1}f'(x)\big)' - x^{\alpha + 1}\big(xf'(x)\big)'\Big) \\
                                     & = \alpha f(x) + x^{\alpha + 1}f'(x).
  \end{align*}
  Finally, we have
  \begin{align*}
    \big[\La, [\Da, \delta_{\alpha}]\big] + \big[\Da, [\delta_{\alpha}, \La]\big] + \big[\delta_{\alpha}, [\La, \Da]\big] = [\Da, \delta_{\alpha}] + [\delta_{\alpha}, \Da] = 0.
  \end{align*}
\end{proof}

The Jacobi identity above is part of the definition of a \textit{Lie algebra}.
The next definitions are taken from \cite{Hall13} as well as \cite{Faraut08}. We
say that a $\R$-vector space $\mathfrak{g}$ with a bilinear map
$[\cdot, \cdot] : \mathfrak{g}\times \mathfrak{g} \to \mathfrak{g}$ is a real
\textit{Lie algebra} if it satisfies the following properties:
\begin{enumerate}
  \item \textit{(Anti-symmetry)} : $[x, y] = -[y, x]$ for all
        $x, y \in \mathfrak{g}$

  \item \textit{(Jacobi identity)} :
        $\big[x, [y, z]\big] + \big[y, [z, x]\big] + \big[z, [y, x]\big] = 0$
        for all $x, y, z \in \mathfrak{g}$.
\end{enumerate}
In that case, the set
$[\mathfrak{g}, \mathfrak{g}] \coloneq \lb [x,y],\ x,y \in \mathfrak{g} \rb$
equipped with $[\cdot, \cdot]$ is a Lie algebra as well. Set
$\mathfrak{g}_{0} \coloneq \mathfrak{g}$ and
\[
  \mathcal{D}^{k+1}(\mathfrak{g}) \coloneq \big[\mathcal{D}^{k}(\mathfrak{g}), \mathcal{D}^{k}(\mathfrak{g})\big],\ k \in \mathbb{N}.
\]
We call a Lie algebra \textit{solvable} if there exists some $k \in \mathbb{N}$
such that $\mathcal{D}^{k}(\mathfrak{g}) = \lb 0 \rb$. It is easy to check that
the vector space spanned by $\La$, $\Da$ and $\delta_{\alpha}$ with respect to
linear combinations and equipped with the commutator is a Lie algebra. Actually,
it is even solvable.

\begin{proposition}
  The vector space
  $\mathfrak{g}_{\alpha} \coloneq \mathrm{span}(\La, \Da, \delta_{a}, \mathrm{Id})$
  equipped with the commutator $[\cdot, \cdot]$ is a solvable Lie algebra.
\end{proposition}
\begin{proof}
  We have already proved that $\mathfrak{g}_{\alpha}$ is a Lie algebra thanks to
  equality \eqref{4_Jcbi}. The fact it is solvable comes from noticing that
  $\mathcal{D}^{1}(\mathfrak{g}_{\alpha}) = [\mathfrak{g}_{\alpha}, \mathfrak{g}_{\alpha}] = \mathrm{span}(\Da, \delta_{\alpha})$,
  so that $\mathcal{D}^{3}(\mathfrak{g}_{\alpha}) = \lb 0 \rb$, thanks to
  identities \eqref{4_C-A}, \eqref{4_inf_comm} and \eqref{4_thrd_comm}.
\end{proof}

Lie algebras have been thoroughly classified, so where does
$\mathfrak{g}_{\alpha}$ sit in that classification? In \cite{MacCallum99}, a
complete classification of 4 dimensional Lie algebras is exposed. Setting:
\begin{align*}
  X_{1} & \coloneq \La                    \\
  X_{2} & \coloneq -\Da + \delta_{\alpha} \\
  X_{3} & \coloneq \Da + \delta_{\alpha}  \\
  X_{4} & \coloneq -2\mathrm{Id},
\end{align*}
the previous commutation relations become
\begin{align*}
  [X_{1}, X_{2}] & = -X_{3}           \\
  [X_{1}, X_{3}] & = -X_{2}           \\
  [X_{2}, X_{3}] & = X_{4}            \\
  [X_{i}, X_{4}] & = 0,\ i = 1, 2, 3.
\end{align*}
This matches the class U310 defined in \cite{MacCallum99} (p. 307), implying
that $\mathfrak{g}_{\alpha}$ is isomorphic to that Lie algebra. Notice also that
if we restrict ourselves to $X_{2}$, $X_{3}$ and $X_{4}$, we get the commutation
relations characteristic of the \textit{Heisenberg algebra}, so that
$\mathfrak{g}_{\alpha}$ contains a subalgebra isomorphic to it. The implications
of those results, if any, remain to study.


In dimension $1$, the application $T_{\alpha}$ is actually a Lie algebra
isomorphism on the Lie algebra spanned by $\mathrm{span}(\Li, \Di, \delta_{1})$,
in the sense that:
\[
  T_{\alpha}[\phi_{1}, \phi_{2}] = \big[T_{\alpha}\phi_{1}, T_{\alpha}\phi_{2}\big],\ \phi_{1},\ \phi_{2} \in \mathrm{span}(\Li, \Di, \delta_{1}).
\]
We thus see that $\delta_{\alpha}$ satisfies
$\delta_{\alpha} = T_{\alpha}\delta_{1}T_{\alpha}^{-1}$.

We conclude this section by defining a Markov process whose semi-group is
$(\Pa_{t})_{t\geq 0}$. It will be expressed in terms of extremal integrals, as
defined in the subsection \ref{prelim_SEI} of the preliminaries.

\begin{definition}
  The \textit{Fréchet process} is defined as :
  \[
    X_{t} \coloneq e^{-\frac{t}{\alpha}}X_{0}\oplus \mxint{0}{t}e^{-\frac{1}{\alpha}(t-s)}\dif M_{\alpha}(s).
  \]
  where $M_{\alpha}$ is a $\alpha$-Fréchet random sup-measure with Lebesgue
  control measure.

\end{definition}
Formally this process is the exact counterpart of the standard
Ornstein-Uhlenbeck semi-group, except that the addition is replaced by the
maximum, and the stochastic integral by the extremal integral.

\begin{proposition}
  The process $(X_{t})_{t\geq 0}$ is a Markov process and
  \[
    \esp\big[f(X_{t}) \/ X_{0} = x\big] = \Pa_{t}f(x),\ x \in \R_{+},\ f \in \cilogi.
  \]
\end{proposition}

\begin{proof}
  Let us note provisionally
  $\widetilde{\Pa_{t}}f(x) \coloneq \esp[f(X_{t}) \/ X_{0} = x]$. It is clear
  that for all non-negative $t$, $\widetilde{\Pa_{t}}$ is a linear operator and
  that $\widetilde{\Pa_{0}} = \mathrm{Id}$. Now we need to check that
  $\widetilde{\Pa_{t}}\circ \widetilde{\Pa_{s}} = \widetilde{\Pa_{t+s}}$:

  \begin{align*}
    (\widetilde{\Pa_{t}}\circ \widetilde{\Pa_{s}})f(x) & = \esp\Big[f\big(e^{-\frac{t}{\alpha}}X_{s} \oplus \mxint{0}{t}e^{-\frac{1}{\alpha}(t-u)}\dif M_{\alpha}(u)\big)\ \big\vert\ X_{0} = x\Big]                                                \\
                                                       & = \esp\Big[f\big(e^{-\frac{1}{\alpha}(t+s)}x \oplus \mxint{0}{t}e^{-\frac{1}{\alpha}(t+s-u)}\dif M_{\alpha}(u) \oplus \mxint{0}{s}e^{-\frac{1}{\alpha}(s-u)}\dif M'_{\alpha}(u)\big)\Big],
  \end{align*}
  where $M_{\alpha}'$ denotes an independent copy of $M_{\alpha}$. Furthermore, by
  the isometry property, we have that:
  \[
    \mxint{0}{s}e^{-(s-u)}\dif M_{\alpha}(u) \overset{\dif}{=} \mathcal{F}\Big(1,\big(\int_{0}^{s}e^{-(s-u)}\dif u\big)^{1/\alpha}\Big) = \mathcal{F}\Big(\alpha,\big(\int_{t}^{t+s}e^{-(t+s-u)}\dif u\big)^{1/\alpha}\Big).
  \]
  Consequently, injecting this result in the previous computation:
  \begin{align*}
    (\widetilde{\Pa_{t}}\circ \widetilde{\Pa_{s}})f(x) & = \esp\Big[f\big(e^{-\frac{1}{\alpha}(t+s)}x \oplus \mxint{0}{t}e^{-\frac{1}{\alpha}(t+s-u)}\dif M_{\alpha}(u) \oplus \mxint{0}{s}e^{-\frac{1}{\alpha}(s-u)}\dif M'_{\alpha}(u)\big)\Big] \\
                                                       & = \esp\Big[f\big(e^{-\frac{1}{\alpha}(t+s)}x \oplus \mxint{0}{t}e^{-(t+s-u)}\dif M_{\alpha}(u) \oplus \mxint{t}{t+s}e^{-\frac{1}{\alpha}(t+s-u)}\dif M'_{\alpha}(u)\big)\Big]             \\
                                                       & = \esp\Big[f\big(e^{-\frac{1}{\alpha}(t+s)}x \oplus \mxint{0}{t+s}e^{-\frac{1}{\alpha}(t+s-u)}\dif M_{\alpha}(u) \big)\Big]                                                               \\
                                                       & = \widetilde{\Pa_{t+s}}f(x).
  \end{align*}
  Furthermore, since $\int_{0}^{t}e^{-(t-s)}\dif s = 1-e^{-t}$, we have that
  $\int_{0}^{t}e^{-\frac{1}{\alpha}(t-u)}\dif M_{\alpha}(u) \overset{\dif }{=} \mathcal{F}\big(\alpha, (1-e^{-t})^{1/\alpha}\big)$,
  so
  \[
    \mxint{0}{t}e^{-\frac{1}{\alpha}(t-s)}\dif M_{\alpha}(s) \overset{\dif }{=} (1-e^{-t})^{\frac{1}{\alpha}}Z,
  \]
  where $Z$ is a random variable with Fréchet distribution $\mathcal{F}(\alpha)$.
  Therefore $(\widetilde{\Pa_{t}})_{t \geq 0} = (\Pa_{t})_{t\geq 0}$.

\end{proof}

\begin{remark}
  Using the terminology of \cite{Stoev05}, this process is an instance of an
  \textit{integral moving maximum process}:
  \[
    X_{t} = \mxint{\R_{+}}{}f(t-u)\dif M_{\alpha}(u),
  \]
  where $f \in \mathbf{L}_{+}^{\alpha}(\R_{+},\lambda)$ and $\lambda$ is the
  Lebesgue measure. Here $f = u \mapsto e^{-\frac{1}{\alpha}u}\ind_{\R_{+}}(u)$.
\end{remark}

We exhibit some sample paths of $X_{t}$ starting at $X_{0} = 3$, for different
values of $\alpha$.
\begin{center}
  \begin{figure}[h!]
    \centering \includegraphics[width=1.0\textwidth]{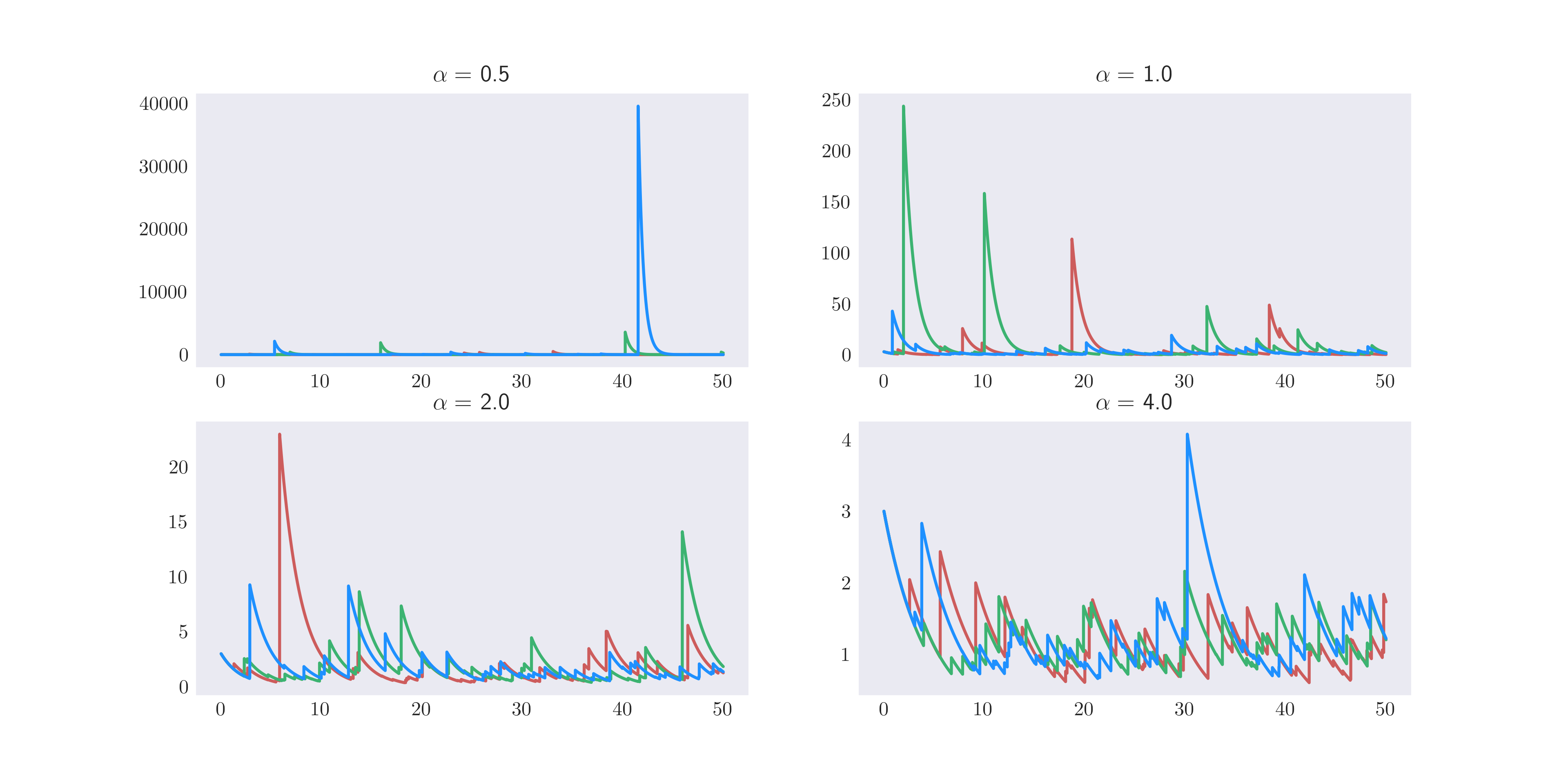}
    \caption{Four paths of $X_{t}$ for
      $\alpha \in \big\lb \frac{1}{2}, 1, 2, 4 \big\rb$}
    \label{ch_2_X_t}
  \end{figure}
\end{center}

The notion of extremal integral allows us to define another stochastic process
of interest here.

\begin{definition}

  Let $Z_{0}$ be a positive random variable and $\alpha$ a positive number. An
  $\alpha$-\textit{max-stable motion} $(Z_{t})_{t\geq 0}$ is a stochastic
  process such that there exists an $\alpha$-Fréchet random measure satisfying
  \[
    Z_{t} = Z_{0}\oplus\mxint{0}{t} 1 \dif M_{\alpha}(s),\ t\geq 0.
  \]
\end{definition}
The next proposition links max-stable motions to the Fréchet process.

\begin{proposition}

  $(Z_{t})_{t\geq 0}$ is a Markov process whose generator $\mathscr{K}_{\alpha}$
  is:
  \[
    \mathscr{K}_{\alpha}f(x)= \int_{x}^{\infty}\big(f(x\oplus r) - f(x)\big)\frac{\alpha}{r^{\alpha + 1}}\dif r = \Da f(x),
  \]
  for $x \in \Rps$ and $f \in \cilogi$.
\end{proposition}

\begin{proof}

  Let $f$ be in $\cilogi$ and define
  $\mathbf{Q}^{\alpha}_{t}f(x) \coloneq \esp[f(Z_{t}) \/ X_{0} = x]$. We will
  contend ourselves with computing the generator of
  $(\mathbf{Q}^{\alpha}_{t})_{t\geq 0}$. The proof of the
  $\L^{2}(\pms_{\alpha})$-convergence of
  $t^{-1}(\mathbf{Q}^{\alpha}_{t} - \mathrm{Id})$ to $\mathscr{K}_{\alpha}$ is
  essentially the same as the one we gave for $\Ln$.

  By definition of $Z_{t}$, we have:
  \begin{align*}
    Z_{t+s} & = Z_{0}\oplus \mxint{0}{t+s} 1 \dif M_{\alpha}(u)                                        \\
            & = Z_{0}\oplus \mxint{0}{s} 1 \dif M_{\alpha}(u)\oplus\mxint{s}{t+s} 1 \dif M_{\alpha}(u) \\
            & \overset{\dif }{=} Z_{s}\oplus tZ,
  \end{align*}
  where $Z$ is a random variable with Fréchet distribution $\mathcal{F}(\alpha)$
  independent of $\sigma(Z_{u},\ u\leq s)$. This proves that $(Z_{t})_{t\geq 0}$
  is a Markov process with semi-group
  \[
    \mathbf{Q}^{\alpha}_{t}f(x) = \esp\big[f(x\oplus tZ)\big] = f(x)e^{-\frac{t}{x^{\alpha}}} + t\int_{x}^{\infty}f(r)e^{-\frac{t}{r^{\alpha}}}\frac{\alpha}{r^{\alpha+1}}\dif r,\ x > 0.
  \]
  An easy calculation yields the generator of the proposition.
\end{proof}

In other words, the generator of the Fréchet process writes as the generator of
the dilation semi-group plus the generator of an $\alpha$-max-stable motion.
This is similar to what is observed for the standard Ornstein-Uhlenbeck process,
where the max-stable motion is replaced by the Brownian motion, or more
generally with $\alpha$-stable Ornstein-Uhlenbeck processes (see
\cite{Samorodnitsky94}).

\printbibliography{}

\end{document}